%% file: TgraphMHS.tex
\documentclass[12pt]{amsart}
\usepackage{verbatim}
\usepackage{amssymb}
\usepackage[retainorgcmds]{IEEEtrantools}
\usepackage[pdftex]{graphicx,color} 
\usepackage{color}
\usepackage{multirow}
\usepackage{hyperref}

\newtheorem{lemma}{Lemma}[section]
\newtheorem{theorem}[lemma]{Theorem}
\newtheorem{corollary}[lemma]{Corollary}
\newtheorem{proposition}[lemma]{Proposition}
\theoremstyle{definition}
\newtheorem{example}[lemma]{Example}
\newtheorem{remark}[lemma]{Remark}

\theoremstyle{remark}
\theoremstyle{definition}

\newtheorem{definition}[lemma]{Definition}

\renewcommand{\geq}{\geqslant}
\renewcommand{\leq}{\leqslant}
\newcommand{\K}{\ensuremath{K}} 
\renewcommand{\AA}{\ensuremath{\mathbb{A}}}

\newcommand{\NN}{\ensuremath{\mathbb{N}}}

\newcommand{\ZZ}{\ensuremath{\mathbb{Z}}} 
\newcommand{\Hc}{\ensuremath{\mathbb{H_{\bf c}}}}
\newcommand{\cH}{\ensuremath{\mathcal{H}}} 
\newcommand{\cM}{\ensuremath{\mathcal{M}}}
\newcommand{\cN}{\ensuremath{\mathcal{N}}}
\newcommand{\cP}{\ensuremath{\mathcal{P}}} 
\newcommand{\cS}{\ensuremath{\mathcal{S}}} 
\newcommand{\cT}{\ensuremath{\mathcal{T}}} 

\newcommand{\cW}{\ensuremath{\mathcal{W}}} 
\renewcommand{\c}{\ensuremath{{\mathbf c}}}

\newcommand{\opp}{\mathrm{opp}}

\newcommand{\<}{\langle}

\DeclareMathOperator{\Hilb}{Hilb}

\DeclareMathOperator{\Hom}{Hom}

\DeclareMathOperator{\lcm}{lcm}

\DeclareMathOperator{\Spec}{Spec}
\DeclareMathOperator{\inn}{in}
\DeclareMathOperator{\Mon}{Mon}
\DeclareMathOperator{\lc}{lc}

\DeclareMathOperator{\rk}{rk}

\begin{document}

\title{The $T$-graph of a multigraded Hilbert scheme}

\author{Milena Hering}
\address{Department of Mathematics \\
University of Connecticut \\
196 Auditorium Road \\
Storrs CT 06269 \\
USA}

\email{milena.hering@uconn.edu}

\author{Diane Maclagan}

\address{Mathematics Institute\\
Zeeman Building\\
University of Warwick\\
Coventry CV4 7AL\\
United Kingdom}

\email{D.Maclagan@warwick.ac.uk}

\dedicatory{Dedicated to the memory of Mikael Passare}

\begin{abstract}
The $T$-graph of a multigraded Hilbert scheme records the zero and
one-dimensional orbits of the $T = (\K^*)^n$ action on the Hilbert
scheme induced from the $T$-action on $\mathbb A^n$.  It has vertices
the $T$-fixed points, and edges the one-dimensional $T$-orbits.  We
give a combinatorial necessary condition for the existence of an edge
between two vertices in this graph.  For the Hilbert scheme of points
in the plane, we give an explicit combinatorial description of the
equations defining the scheme parameterizing all one-dimensional torus
orbits whose closures contain two given monomial ideals.  For this
Hilbert scheme we show that the $T$-graph depends on the ground field,
resolving a question of Altmann and Sturmfels.
\end{abstract}

\maketitle

\section{Introduction}

The Hilbert scheme of points in the plane is a classical and
well-studied space, and an important technique in its study is to
consider the fixed points of the action of the torus $(\K^*)^2$ on the
Hilbert scheme.  However there is no known combinatorial condition
deciding whether two fixed points, which correspond to monomial ideals
in $\K[x,y]$, lie in the closure of a one-dimensional torus orbit. In
this paper we give a necessary combinatorial condition in the more
general context of multigraded Hilbert schemes.

The multigraded Hilbert scheme, introduced by Haiman and Sturmfels in
\cite{HaimanSturmfels}, parameterizes subschemes $Z$ of $\mathbb A^n$
invariant under the action of an abelian group for which $H^0(\mathcal
O_Z)$ has a prescribed decomposition into irreducible representations.
Equivalently, $\Hilb^h_S$ parameterizes all ideals $I$ in the
polynomial ring $S=\K[x_1,\dots,x_n]$ that are homogeneous with
respect to a grading by an abelian group $A$, and have a fixed
multigraded Hilbert function $h: A \rightarrow \mathbb N$ given by
$h(a)=\dim_{\K}(S/I)_a$.  Specific examples of multigraded Hilbert
schemes include the Grothendieck Hilbert scheme of subschemes of
projective space, Hilbert schemes of points in affine space, and
$G$-Hilbert schemes for abelian groups $G$.

The action of $T = (\K^*)^n$ on $\mathbb A^n$ induces an action of $T$
on $\Hilb^h_S$ whose fixed points are the monomial ideals in
$\Hilb^h_S$.  The $T$-graph of the multigraded Hilbert scheme
$\Hilb^h_S$ has vertices these fixed points, and an edge between two
vertices $M$ and $N$ if there is a one-dimensional torus orbit whose
closure contains $M$ and $N$.  This is closely related to the
\emph{graph of monomial ideals} of Altmann and
Sturmfels~\cite{AltmannSturmfels}.

The main result of this paper is a necessary condition
(Theorem~\ref{thm:main}) for two
vertices $M$ and $N$ to be connected by an edge of the $T$-graph.  In
the case of $\Hilb^d(\mathbb A^2)$ we give a combinatorial description
(Theorem~\ref{t:combinatorial}) over $\mathbb Z$ for the equations of
the edge-schemes describing all $T$-orbits joining a pair of fixed
points.

One motivation to study the $T$-graph is to understand the
connectedness of multigraded Hilbert schemes.  In contrast to the
classical Hilbert scheme of subschemes of projective space, which is
always connected \cite{Hartshorne}, multigraded Hilbert schemes can be
disconnected \cite{Santos}.  However, a necessary and sufficient
condition for a multigraded Hilbert scheme to be connected (when the
grading is positive and $\K=\mathbb C$) is for the $T$-graph of
$\Hilb^h_S$ to be connected; see \cite[Corollary
  16]{AltmannSturmfels}.  The sufficiency has been well-exploited in
the literature (see \cite{PeevaStillmanBinomial},
\cite{MaclaganSmithHilbert}), and we hope that through a better
understanding of the $T$-graph the necessity can be used to exhibit
more tractable examples of disconnected multigraded Hilbert schemes.
Another motivation comes from the use of $T$-graphs of varieties to
understand cohomology.  The standard set-up of \cite{GKM} to compute
cohomology from the $T$-graph of a variety requires that the
one-dimensional orbits be isolated, which need not be the case for
multigraded Hilbert schemes.  However one could still hope to deduce
information about the cohomology in these cases; see for example
\cite{BradenChenSottile, EvainChow}.

Monomial ideals are fundamentally combinatorial objects, and a natural
question is whether the $T$-graph has a purely combinatorial
description.  The main results of this paper illustrate the complexity
of this question.  Since the one-dimensional orbits are not isolated,
we consider the edge-scheme $E(M,N)$ parameterizing all ideals $I \in
\Hilb^h_S$ lying in a one-dimension $T$-orbit whose closure contains
$M$ and $N$.  
In Example~\ref{ex:almostall1} we construct an
example of an edge-scheme in $\Hilb^{10}(\mathbb A^2)$ that has
$\mathbb R$-valued points, but no $\mathbb Q$-valued points.  This
shows that the $T$-graph depends on the field $\K$, solving a problem
posed by Altmann and Sturmfels in \cite[Section 5]{AltmannSturmfels}.
It also shows that there cannot be a purely combinatorial description of the
generators of an ideal $I$ contained in a one-dimensional $T$-orbit.

Our first step towards a combinatorial necessary condition for the
existence of an edge in the $T$-graph is to show that we can reduce to
a simpler multigraded Hilbert scheme whose Hilbert function has finite
support,  so $\sum_{a \in A} h(a) < \infty$.  More precisely, we show that if there is an edge in the
$T$-graph of $\Hilb^h_S$ between two monomial ideals $M$ and $N$ then
there exists a positive grading of the polynomial ring $S$ by $\mathbb
Z^n/\mathbb Z\c$ for some $\c \in \mathbb Z^n$, and a Hilbert function
$H : \mathbb Z^n/\mathbb Z \c \rightarrow \mathbb N$ such that $M$ and
$N$ have Hilbert function $H$, and there is an edge in the $T$-graph
between $M$ and $N$ in this refined multigraded Hilbert scheme, which
we denote by $\Hc(H)$.  See Corollary~\ref{c:reductionstep}.

The following theorem, which holds over an arbitrary base,
gives the reduction to finite support Hilbert functions.

\begin{theorem} \label{t:isom}
Let $h : A \rightarrow \mathbb N$ be a Hilbert function.  If the
$A$-grading of $S$ is positive, then there exists $\overline{h} : A
\rightarrow \mathbb N$ with $\sum_{a \in A} \overline{h}(a)<\infty$ and an isomorphism
\begin{equation*}\label{eq:isom}
	\Hilb^h_S \cong \Hilb^{\overline{h}}_S.
\end{equation*}
This isomorphism respects the $T$-action on the two Hilbert schemes.
\end{theorem}

Our combinatorial necessary condition uses the following definition of an arrow map.
It is a modification of the definition of
a ``system of arrows'' introduced by Evain in
\cite{EvainIrreducible} to study incidence conditions for
Bia{\l}ynicki-Birula cells in multigraded Hilbert schemes of
points in the plane; see Remark~\ref{r:incidence}.

\begin{definition}\label{d:arrowmap}
Let $S$ be graded by $\ZZ^n/\ZZ\c$ for some $\c \in \ZZ$, and let
$\prec$ be a monomial term order on $S$. For a monomial ideal $M$, let
$\Mon(M)$ denote the set of monomials in $M$.  If two monomials
$m=x^{u}$ and $m'=x^{v}$ have the
same degree, then $u-v = \ell \mathbf{c}$, and we
define the \emph{distance} between $m$ and $m'$ to be $d(m,m') = |\ell|$. 
For two monomial ideals $M$ and
$N$, we say that $f \colon \Mon(M) \to \Mon(N)$ is an \emph{arrow map}
if
\begin{enumerate}
\item \label{e:bijectivedecreasing}  $f$ is a degree-preserving bijection
	such that 
$m \succeq f(m)$ for all $m\in \Mon(M)$; 
\item  \label{e:mshorter}
for all $m \in \Mon(M)$ and all multiples $m'$ of $m$, we have 
$$d ( m', f(m'))\leq d(m,f(m));$$
\item \label{e:mprimeshorter} for all $m \in \Mon(N)$ and all
  multiples $m'$ of $m$, we have
$$d(f^{-1}(m'),m')\leq d(f^{-1}(m),m).$$
\end{enumerate}
\end{definition}

See Example \ref{e:arrowmaps} and Figure \ref{f:arrowmap1} for an
illustration of this concept.  

\begin{theorem}\label{thm:main}
	Assume that $H \colon \mathbb Z^n/\ZZ \c \to \mathbb N$ has finite
        support, and let $M, N$ be monomial ideals in $\Hc(H)$ that
        are connected by an edge in the $T$-graph of $\Hc(H)$.
	\begin{enumerate}
		\item \label{thm:main1}
	There exists an arrow map $f \colon M \to N$ 
        with respect to 
	some term order $\prec$.
\item \label{thm:main2}
	 Fix
	 $r_1,\dots,r_n$  such that $x_i^{r_i}\in M
	 \cap N$
  for all $1\leq i\leq n$, and let  $Q = \langle x_1^{r_1}, \dots,
  x_n^{r_n}\rangle$.
  Then there also exists an arrow map $\widehat{f} \colon \Mon((Q\colon 
  M)) \to \Mon((Q\colon
  N))$ with respect to the same  term order as in
  \eqref{thm:main1}.
\end{enumerate}
\end{theorem}

Condition \eqref{thm:main1} holds without the condition that the
Hilbert function has finite support; see Corollary~\ref{cor:main}.  It is not
sufficient for the existence of an edge in the $T$-graph; see Example
\ref{ex:all}.  We do not know, however,  of an example showing that both
\eqref{thm:main1} and \eqref{thm:main2} together do not suffice to
guarantee the existence of an edge.  For the Hilbert scheme of $d$
points in the plane, these conditions are sufficient for $d\leq 16$;
see Table~\ref{table}.  On the other hand, the proof of
Theorem~\ref{thm:main} is based on associating an arrow map to an
ideal $I$ in the $T$-orbit  (Proposition \ref{thm:arrowmapnecessary}), and
we have examples of arrow maps that are not associated to ideals
(Example \ref{ex:almostall}).

In the case of the Hilbert scheme of points
in the plane, there exists an explicit combinatorial description
of the equations 
for the edge-scheme describing 
all one-dimensional $T$-orbits joining a
fixed pair of monomial ideals. In particular, 
this scheme is defined over $\ZZ$.

\begin{theorem} \label{t:combinatorial}
Let $M,N$ be monomial ideals in $\K[x,y]$ with the same Hilbert
function with respect to a positive $\mathbb Z^2/\mathbb
Z\mathbf{c}$-grading.  The ideal of the edge-scheme $E(M,N)$ is
generated by polynomials $F_{n,s}$ with integer coefficients, where
$n$ is a minimal generator of $N$, and $s$ is a standard monomial of
$M$ with $\deg(s)=\deg(n)$. 
\end{theorem}

The terms of the polynomials $F_{n,s}$ have an explicit
combinatorial form in terms of the torus weights of the 
action of the torus on the tangent spaces to $M$ and $N$ 
in $\Hc(H)$ that we describe in detail in Section
\ref{s:utahequations}.  The equations are obtained by combining an algorithm
of Altmann and Sturmfels \cite[Algorithm 5]{AltmannSturmfels} with a
description, due to Evain~\cite{EvainIrreducible}, of the
Bia{\l}ynicki-Birula cells in this Hilbert scheme.

This paper is partially experimental in nature, and we relied heavily
on computations using the computer algebra system Macaulay 2
\cite{M2}.  The resulting code is available from the second author's 
webpage
as the Macaulay 2 package {\tt TEdges} \cite{TEdges}.  Some details of
these computations are given in the last section of the paper.

{\bf Acknowledgments}
We thank Bernd Sturmfels for stimulating our interest in the  
$T$-graph.  This paper was written at several
mathematical institutes and we are grateful to the Institute of
Mathematics and its Applications, the Mathematical 
Sciences Research Institute, the Mathematisches
Forschungsinstitut Oberwolfach, and the Mittag Leffler institute for
their hospitality.  The first author was partially supported by an
Oberwolfach Leibniz Fellowship and NSF grant DMS 1001859.   The second author was partially supported by EPSRC grant EP/I008071/1.

\section{Reduction to the positively graded artinian case}

\label{s:MHS}

In this section we show that the study of the $T$-graph of arbitrary
multigraded Hilbert schemes can be reduced to the study of multigraded
Hilbert schemes parameterizing finite-length ideals that are homogeneous
with respect to a positive grading by $\ZZ^n/\ZZ\c$, where $\c \in
\ZZ^n$.  Moreover, we show that every positively-graded multigraded
Hilbert scheme is isomorphic to some multigraded Hilbert scheme
parameterizing finite-length ideals.  Throughout this section $S$ denotes
the polynomial ring $S=\K[x_1,\dots,x_n]$, where unless otherwise
noted $\K$ is a field.

\subsection{Multigraded Hilbert schemes}\label{s:positivereduction}

\begin{definition}[\cite{HaimanSturmfels}]
Fix a grading by an abelian group $A$ on $S$, and a function $h : A
\rightarrow \mathbb N$.  The multigraded Hilbert scheme $\Hilb^h_S$
parameterizes all homogeneous ideals $I$ in $S$ with Hilbert function
$\dim_{\K} (S/I)_a = h(a)$ for all $a \in A$.
\end{definition}

The multigraded Hilbert scheme $\Hilb^h_S$ is a quasiprojective scheme
over $\Spec(\K)$; see \cite[Theorem 1.1]{HaimanSturmfels}. 
Following \cite{HaimanSturmfels}, we say that the grading 
is {\em positive} 
if $\dim_{\K} S_a < \infty$ for all $a \in
A$. In this case  $\Hilb^h_S$ is projective; see \cite[Corollary
  1.2]{HaimanSturmfels}.  

\begin{example}
\begin{enumerate}
	\item  Fix an integer-valued polynomial $P$. 
		There exists $D\gg0$ such that when 
  $A=\mathbb Z$, $h(a)=0$ for $a<0$, $h(a)=\dim_{\K} S_a$ for $0
  \leq a <D$, and $h(a) =P(a)$ for $a \geq D$, then $\Hilb^h_S$ is
  Grothendieck's Hilbert scheme $\Hilb^P(\mathbb P^{n-1})$
  parameterizing all subschemes of $\mathbb P^{n-1}$ with Hilbert
  polynomial $P$ \cite[\S 4]{HaimanSturmfels}.

\item When $A=0$, and $h(0)=d$,  
then $\Hilb^h_S$ is the Hilbert scheme
  $\Hilb^d(\mathbb A^n)$ of $d$ points in $\mathbb A^n$.

\item For arbitrary $A$, if $h(a)=1$ whenever $\dim_{\K}S_a >0$ and
  $h(a)=0$ otherwise, then $\Hilb^h_S$ is the toric Hilbert scheme
  \cite{PeevaStillman}, \cite[\S 5]{HaimanSturmfels}.  When $A$
  is finite this is Nakamura's $G$-Hilbert scheme~\cite{Nakamura}.

\end{enumerate}

\end{example}

\subsection{Background on the $T$-graph of a multigraded 
Hilbert scheme}

The action of the torus $T =(\K^*)^n$ on $\mathbb A^n$ induces an 
action on $\Hilb^h_S$ whose fixed points 
are the monomial ideals contained in $\Hilb^h_S$.

\begin{definition}
The $T$-graph of $\Hilb^h_S$ has vertices the monomial ideals in
$\Hilb^h_S$.  There is an edge joining two monomial ideals $M, N \in
\Hilb^h_S$ if there is $I \in \Hilb^h_S$ such that the $T$-orbit of
$I$ contains $M$ and $N$ in its closure and is one-dimensional.
\end{definition}

The $T$-graph has an interpretation in terms of Gr\"obner theory,
which we now explain.  For basic facts about Gr\"obner bases and
initial ideals, see \cite{CoxLittleOshea}.  For the geometric
interpretation of initial ideals as limits of one-parameter torus orbits see
\cite[Chapter 15.8]{Eisenbud}.

For $\c=(c_1, \ldots, c_n) \in \ZZ^n$, we define a
$(\ZZ^n/\ZZ\c)$-grading on $S=\K[x_1, \ldots, x_n]$ by letting
$\deg(x_i) = \mathbf{e_i} + \ZZ\c$, where $\{\mathbf{e_1}, \ldots,
\mathbf{e_n}\}$ is the standard basis of $\ZZ^n$. 
Let $\mathbf{c}^+ = \sum_{c_i>0} c_i \mathbf{e}_i$ and
$\mathbf{c}^- = \sum_{c_i<0} -c_i \mathbf{e}_i$, so
$\mathbf{c}=\mathbf{c}^+-\mathbf{c}^-$. The grading induced by $\c$ is
positive if and only if $\c^+\neq 0$ and $\c^- \neq 0$.

We next note that any non-monomial ideal $I$ that is homogeneous with
respect to this $\ZZ^n/\ZZ\c$-grading has either exactly two initial
ideals, if the grading is positive, or exactly one
initial ideal otherwise.  Indeed, a homogeneous polynomial has the
form $f=\sum_{i=0}^s a_i x^{u+i\mathbf{c}}$, where
we assume $a_0,a_s \neq 0$. The initial term is
$\inn_{\prec}(f)=a_0x^{u}$ if
$x^{\mathbf{c}^+} \prec x^{\mathbf{c}^-}$ and
$\inn_{\prec}(f)=a_s x^{\mathbf{u}+s\mathbf{c}}$ if
$x^{\mathbf{c}^+} \succ x^{\mathbf{c}^-}$.  Thus the
initial ideal of $I$ with respect to a term order $\prec$ only depends
on whether $x^{\mathbf{c}^+} \prec x^{\mathbf{c}^-}$
or $x^{\mathbf{c}^+} \succ x^{\mathbf{c}^-}$.  If
the $\ZZ^n/\c\ZZ$-grading is positive, then both $\c^+$ and $\c^-$ are
nonzero, so a non-monomial ideal has exactly two monomial initial
ideals.  However, if the grading is not positive, so without loss of
generality $\c \geq 0$, the monomials of degree $a$ have the form
$x^{u+i \c}$ for some $u \in \mathbb N^n$
and $i \geq 0$, and the standard monomials of any monomial initial
ideal in degree $a$ are $x^{u+i\c}$ for $ 0 \leq i <
\dim_{\K}(S/I)_a$.  Thus for every Hilbert function $H$ there exists
exactly one monomial ideal with this Hilbert function. In particular,
a homogeneous polynomial has exactly one initial ideal in this case.

\begin{definition}\label{d:opp}
Assume that the $\ZZ^n/\ZZ\c$-grading on $S$ is positive. 
For a term order $\prec$ with $x^{\c^+} \prec x^{\c^-}$ (resp. $x^{\c^+}
\succ x^{\c^-}$) we let $\prec^{\opp}$ be any term order with $x^{\c^+}
\succ x^{\c^-}$ (resp. $x^{\c^+} \prec x^{\c^-}$).  
\end{definition}

\begin{proposition}
	\label{p:torusgrading}
	Let $M,N$ be monomial ideals in $S$. 
	 There exists a one-dimensional torus orbit
         $\mathcal{O}\subset \Hilb^h_S$ such that
         $\overline{\mathcal{O}} = \mathcal{O} \cup \{M,N\}$ if and
         only if there exists $\c = (c_1, \ldots, 
	 c_n)\in \ZZ^n$ with $\c^+ \neq 0$, 
	 $\c^- \neq 0$, a term order $\prec$, and
         an ideal $I$ homogeneous with respect to the
         $\ZZ^n/\ZZ\c$-grading such that $\inn_{\prec}I =M$ and
         $\inn_{\prec^{\opp}}I=N$.  
The vector $\c$ can be chosen so that the $\mathbb Z^n/ \mathbb
Z \c $-grading of $S$ refines the grading on $S$.
 \end{proposition}
\begin{proof}  

 First note that an ideal $I$ is contained in a one-dimensional torus
 orbit if and only if $I$ is fixed by a codimension-one subtorus $T'$.
 If $T'$ is a codimension-one subtorus of $T$ and $\c$ is the
 generator of the subgroup of $M:=\Hom(T,\K^*)\cong \ZZ^n$ that is the
 image of the inclusion $\Hom(T/T', \K^*) \cong \ZZ \hookrightarrow
 \Hom(T,\K^*)$, then $I$ is fixed by $T'$ if and only if $I$ is
 homogeneous with respect to the induced $\mathbb Z^n/\c \mathbb
 Z$-grading; see \cite[Lemma 10.3]{MillerSturmfels}.

For $I \in \Hilb^h_S$ lying in a one-dimensional torus orbit, $I$ is
also homogeneous with respect to the $A$-grading on $S$.  Write $A =
\mathbb Z^n/L$ for some lattice $L$; two monomials $x^u$ and $x^v$ have
the same degree with respect to the $A$-grading if and only if $u-v \in
L$.  Choose a generating set for $I$ that is homogeneous with respect to
both the $\mathbb Z^n/\mathbb Z \mathbf{c}$ and $A$ gradings, with the property
that no summand of any generator lives in $I$.  Such generators have the
form $\sum a_i x^{u + i \mathbf{c}}$, where $i\mathbf{c} \in L$.  If $j$
the greatest common divisor of all differences $i-i'$ with $a_i, a_{i'}
\neq 0$, then $I$ is homogeneous with respect to the $\mathbb Z^n/\mathbb
Z j\mathbf{c}$-grading, and $j \mathbf{c} \in L$.  Thus after replacing
$\mathbf{c}$ by $j\mathbf{c}$ the grading of $S$ by $\mathbb Z^n/\mathbb Z
\c$ refines the existing grading in the sense of
\cite[p729]{HaimanSturmfels}.  It thus remains to check that $M$ and $N$ are
the two initial ideals of $I$.

Let $H$ be the Hilbert function of $I$ with respect to the $\mathbb
Z^n/\c \mathbb Z$-grading.  The inclusion of $\Hilb^H_S$ into
$\Hilb^h_S$ as a closed subscheme (\cite[Proposition
  1.5]{HaimanSturmfels}) means that $M$ and $N$ also have Hilbert
function $H$ with respect to the $\mathbb Z^n/\c \mathbb Z$-grading.
This means that this grading is positive, as 
otherwise there would be only one monomial ideal with Hilbert function
$H$.  Thus $I$ has two initial ideals, so it remains to observe that
all initial ideals of $I$ are contained in the closure of the
$T$-orbit of $I$.

Let $\mathbf{w}\in \Hom(M,\ZZ)$ and let $\lambda_{\mathbf{w}}
\hookrightarrow T$ be the one-parameter subgroup associated to
$\mathbf{w}$.  The composition of the inclusion of the one-parameter
subgroup $\lambda_{\mathbf{w}}$ into $T$ and the projection $T\to
T/T'$ is an isomorphism if and only if $\langle \mathbf{ w},
\mathbf{c}\rangle \neq 0$. So if $I$ is an ideal that is fixed by $T'$
and $\mathbf{w}$ satisfies $\langle \mathbf{w}, \mathbf{c} \rangle
\neq 0$, then the orbits $T\cdot I$ and $\lambda_{\mathbf{w}}\cdot I$
are equal.  In particular, their closures in $\Hilb^h_S$ agree.  The
claim now follows from the interpretation of initial ideals as flat
limits of one-parameter torus orbits, as the two points $M$ and $N$ in
the closure of the $T$-orbit of $I$ must be $\inn_{\mathbf{w}}(I)$ and
$\inn_{\mathbf{-w}}(I)$ in the notation of \cite[Chapter
  15.8]{Eisenbud}. Since $\inn_{\mathbf{w}}(I)$ and
$\inn_{-\mathbf{w}}(I)$ are distinct monomial ideals, they equal
$\inn_{\prec}(I)$ and $\inn_{\prec^{\opp}}(I)$ for some term order
$\prec$.
\end{proof}

It follows that in 
order to study one-dimensional torus orbits in any 
multigraded Hilbert scheme $\Hilb^h_S$, it 
suffices to study multigraded Hilbert schemes 
with grading group $\ZZ^n/\ZZ\c$ and Hilbert function 
$H \colon \ZZ^n/\ZZ\c \to \NN$. We denote 
the corresponding multigraded Hilbert 
scheme by $\Hc(H)$. 

 \begin{corollary} \label{c:reductionstep}
	 Let $M,N$ be monomial ideals in $\Hilb^h_S$.  Then 
         $M$ and $N$ are connected by an edge in the $T$-graph
         if and only if
         there exists $\c\in \ZZ^n$ and $H\colon \ZZ^n/\ZZ\c \to \NN$
         such that $M, N \in \Hc(H)$ and there is an edge between $M$ and $N$ in the
         $T$-graph of $\Hc(H)$.
 \end{corollary}

Note that every $I \in \Hc(H)$ is either a monomial ideal, or lies in a one-dimensional torus orbit.  

\begin{remark}\label{r:tangentweights}
	For a given monomial ideal $M \in \Hilb^h_S$, there are only
        finitely many $\c$ such that $M$ is contained in a
        positive-dimensional $\Hc(H)$ for some Hilbert function
        $H$. These vectors $\c$ are the weights of the torus
        action on the tangent space to $\Hilb^h_S$ at $M$.
\end{remark}

     Recall the
definition of an arrow map (Definition \ref{d:arrowmap}).

\begin{definition} \label{d:posetorder}
       We define a partial order on the monomial ideals in 
       $\Hc(H)$ by letting $M>N$ if there exists a 
       map $f \colon \Mon(M) \to \Mon(N)$ satisfying 
       condition
       \eqref{e:bijectivedecreasing}
       of Definition \ref{d:arrowmap}.
\end{definition}

This partial order was used in Yam\'eogo \cite{Yameogo1, Yameogo2} and
Evain \cite{EvainSchubert} to study a related incidence
question. See Remark~\ref{r:incidence} for a more detailed discussion.

\begin{definition} \label{d:EMN}
For $\c \in \ZZ^n$, $H$ a Hilbert function, 
and a fixed term order $\prec$, let 
$$C_{\prec}(M) = \{ I \in \Hc(H) \mid \inn_{\prec} I = M\}.$$ This is
naturally a subscheme of $\Hc(H)$.  Its equations can be derived from
the Buchberger algorithm for computing Gr\"obner bases.
	For $M$, $N$ monomial ideals in $\Hc(H)$ such that $M>N$ in
        the partial order of Definition~\ref{d:posetorder} we define
        the \emph{edge-scheme} between $M$ and $N$ to be the
        scheme-theoretic intersection
\[ E(M,N):= C_{\prec}(M) \cap C_{\prec^{\opp}}(N).\]
\end{definition}

Altmann and Sturmfels
give an algorithm to compute the edge-scheme  in \cite[Algorithm
  5]{AltmannSturmfels}.

\begin{remark}
	In \cite{AltmannSturmfels} the scheme $C_{\prec}(M)$ is called
          the Schubert scheme $\Omega_c(M)$ in the case that
          $\bf{x}^{\bf{c^+}}\prec \bf{x}^{\bf{c^-}}$.  Choosing a
          suitable isomorphism of $T/T'$ with $\K ^{*}$,
          $C_{\prec}(M)$ consists of all points $I\in \Hc(H)$ such
          that $\lim _{t\to 0}t\cdot I= M$.  In particular, if
          $\Hc(H)$ is smooth, $C_{\prec}(M)$ is the
          Bia{\l}ynicki-Birula cell associated to the fixed point $M$.
\end{remark}

If $\K$ is algebraically closed then $E(M,N)$ is nonempty if and only
if there is an edge in the $T$-graph joining $M$ and $N$.  If $\K$ is not algebraically
closed, the ``only if'' can fail, as we require the existence of a
$\K$-rational point $I$ in the subscheme $E(M,N)$ for there to be an
edge between $M$ and $N$ in the $T$-graph.  This is illustrated in the
following example, which solves a problem of Altmann and
Sturmfels~\cite[Section 5]{AltmannSturmfels}.

\begin{example}
\label{ex:almostall1}
Let $S=\K[x,y]$ be graded by $\mathbb Z^2/\mathbb Z(1,-1)$, so
$\deg(x)=\deg(y)=1$.  Let $M=\<y^5, x^2\rangle$ and $N = \<y^2,
x^5\rangle$. Then the edge-scheme $E(M,N)$ is the subscheme of
$\mathbb A^4$ defined by the ideal $\langle a^4-3a^2b+b^2, c-ad,
1-bd\rangle$, and ideals corresponding to points in the edge-scheme
are given by
	$$ I = \< y^2 +axy + bx^2, x^5\rangle = \< y^5,
x^2+cxy+dy^2\rangle.$$ This can be
computed using the algorithm of \cite[Algorithm 5]{AltmannSturmfels},
or the description given in Section~\ref{s:utahequations}.  This 
scheme is reducible: $a^4 -3a^2b+b^2 = (a^2 - \frac{3+ \sqrt{5}}{2}b)
(a^2 - \frac{3-\sqrt{5}}{2}b)$. It follows from this factorization
that $E(M,N)$ has $\mathbb R$-valued points, but no $\mathbb Q$-valued
points. In particular, this example shows that the $T$-graph of
$\Hilb^{10}(\mathbb A^2)$ depends on the field $\K$.
\end{example}

\subsection{Reduction to the Artinian case}

In this section we prove Theorem~\ref{t:isom}.  This is the only part
of the paper to require details from \cite{HaimanSturmfels}.
Theorem~\ref{t:isom} is only needed in this paper to apply
Theorem~\ref{thm:main} \eqref{thm:main2} in the case where the Hilbert
function does not have finite support, but may be of wider interest.

The ring $\K$ can here be an arbitrary commutative ring; in particular
 $\K=\mathbb Z$ is possible.  We restrict our attention to ideals $I
 \subseteq S$ for which $S/I$ is a locally-free $\K$-module.  By the
 Hilbert function of a homogeneous ideal $I \subseteq S$ with $S/I$ a
 locally-free $\K$-module, we mean the function $A \rightarrow \mathbb
 N$ given by $a~\mapsto~\rk_{\K}(S/I)_a$.

We first recall the construction of the multigraded Hilbert scheme in the
positive-graded case.  The key idea is to restrict to a finite set of
degrees $D$, and consider the Hilbert scheme $\Hilb_{S_D}^h$, which
parameterizes all locally-free $\K$-modules $\oplus_{a \in D} T_a$
with $\rk_{\K}(T_a) = h(a)$ with the property that for all $a,b \in D$
there is a multiplication map $S_{b-a} \times T_a \rightarrow T_b$.
Particular examples of such $\K$-modules are $\oplus_{a \in D}
(S/I)_a$, where $I$ is an ideal with Hilbert function $h$. A major
step in the construction of the multigraded Hilbert scheme 
is to show that for suitably chosen $D$ we
have $\Hilb_{S_D}^h \cong \Hilb_S^h$.

Recall from \cite[Section 3]{HaimanSturmfels} that a finite subset $D$
of the abelian group $A$ is called {\em very supportive} for a Hilbert
function $h:A \rightarrow \mathbb N$ if it satisfies the following
three conditions:

\begin{itemize}
\item[$(g)$] Every monomial ideal with Hilbert function $h$ is
  generated by monomials whose degrees belong to $D$;
\item [$(h)$] Every monomial ideal $M$ whose generators have degrees in $D$
  has the property that if $M$ has Hilbert function $h(a)$ in degree
  $a$ for all $a \in D$, then $M$ has Hilbert function $h$ everywhere;
  and
\item [$(s)$] For every monomial ideal $M$ with Hilbert function $h$,
  the syzygy module of $M$ is generated by syzygies coming from
  relations $x^ux^{v'} =x^vx^{u'} = \lcm(x^u,x^v)$ among generators
  $x^u, x^v$ of $M$ such that $\deg(\lcm(x^u,x^v)) \in D$.
\end{itemize}

Theorem~3.6 of \cite{HaimanSturmfels} says if $D \subset A$ is very
supportive, then $\Hilb_{S_D}^h \cong \Hilb_S^h$, and
\cite[Proposition 3.2]{HaimanSturmfels} implies that such sets exist
for any grading. 

Note that for every positive grading by an abelian group $A$ there
exists a group homomorphism $\phi : A \rightarrow \mathbb Z$ with
$\phi(a)> 0$ whenever $\rk S_a>0$ and $a\neq 0$.

\begin{lemma} \label{l:verysupportive}
Suppose the $A$-grading of $S$ is positive, so there exists a group
homomorphism $\phi : A \rightarrow \mathbb Z$ with $\phi(a)>0$
whenever $\rk S_a >0$ and $a \neq 0$.  Let $D$ be a very supportive
set for $h$, and choose $N>0$ with the property that $\phi(a)<N$ for
all $a \in D$.  Define $\overline{h} : A \rightarrow \mathbb N$ by
$$\overline{h}(a) = \left\{ \begin{array}{ll} 
	 h(a) & \text{if }\phi(a)<N, \\
	 0 & \text{otherwise}.\\
\end{array}
\right.$$ Let $D' = D \cup \{ a : N \leq \phi(a) \leq B(N) \}$, where
every degree $a$ of a generator or minimal syzygy of the monomial ideal
$\langle x^u : \phi(\deg(x^u)) \geq N \rangle$ has $\phi(a) \leq B(N)$.  Then $D'$ is a very supportive set for $\overline{h}$.
\end{lemma}

\begin{proof}
Let $M$ be a monomial ideal with Hilbert function $\overline{h}$, and
let $M'$ be the ideal generated by those monomials in $M$ whose
degrees are contained in $D$.  Then the Hilbert function of $M$ and
$M'$ agree for degrees in $D$ by construction. Since $D$ is very
supportive for $h$, and $M'$ is generated in degrees in $D$, by
property $(h)$ for $D$ the monomial ideal $M'$ has Hilbert function
$h$ everywhere.  This means that $M'_a=M_a$ when $\phi(a)<N$.

Let $M''=M' + \langle x^u : N \leq \phi(\deg(x^u)) \leq B(N)
\rangle$.  We claim that $M''=M$, which shows that $D'$ satisfies
condition $(g)$.  Indeed, since the grading is positive, $M''_a=M'_a$, and thus $M''_a=M_a$, when $\phi(a)<N$. 
 By the definition of $B(N)$, we have \[\langle
x^u : N \leq \phi(\deg(x^u)) \leq B(N) \rangle = \langle x^u :
N \leq \phi(\deg(x^u)) \rangle,\] so $M''_a=S_a$ when $\phi(a)\geq N$, and thus $M''_a=M_a$ when $\phi(a) \geq N$.    Thus $M''=M$ as required.

Since $D$ is very supportive for $h$, all syzygies between generators
of the ideal $M'$ are in degrees in $D$.  By the construction of $B(N)$, all minimal 
syzygies between generators of $\langle x^u : N \leq \phi(\deg(x^u))
\rangle$ have degrees in $D'$.  Finally, if $x^u$ is a minimal
generator of $M'$ and $x^v$ is a minimal generator of $M$ whose degree
is not in $D$, then there is some multiple $x^{u'}$ of $x^u$ that is a
minimal generator of $\langle x^u : N \leq \phi(\deg(x^u)) \rangle$.
The degree $a$ of the syzygy between $x^u$ and $x^v$ has $N \leq
\phi(a) \leq \phi(b)$, where $b$ is degree of the syzygy between
$x^{u'}$ and $x^v$, so the fact that $a \in D'$ follows from the fact
that $b \in D'$.  Thus $M$ satisfies condition $(s)$.

Suppose now that $M$ is a monomial ideal whose generators have degrees
in $D'$, and for which the Hilbert function of $M$ agrees with
$\overline{h}$ for degrees in $D'$.  As before, let $M'$ be the ideal
generated by those monomials in $M$ with degrees belonging to $D$.
Since $D$ is very supportive for $h$, and $M'$ has Hilbert function
$h$ in degrees in $D$, $M'$ has Hilbert function $h$, and thus
$M'_a=M_a$ when $\phi(a)<N$.  Since $\overline{h}(a)=0$ whenever the
$N \leq \phi(a) \leq  B(N)$, $M$
contains $\langle x^u : N \leq \phi(\deg(x^u)) \leq B(N) \rangle =
\langle x^u : N \leq \phi(\deg(x^u)) \rangle$, so $M_a=S_a$ when $\phi(a) \geq N$.
Thus $M$ has Hilbert function $\overline{h}$, so condition $(h)$ is
satisfied.
\end{proof}

We can now prove Theorem~\ref{t:isom}.

\begin{proof}[Proof of Theorem~\ref{t:isom}]
Fix a very supportive set $D$ for $h$.  By \cite[Theorem
  3.6]{HaimanSturmfels} we have $\Hilb^h_{S_D} \cong \Hilb_S^h$.
Choose a group homomorphism $\phi: A \rightarrow \mathbb Z$ with $\phi(a)>0$ whenever
$\rk S_a >0$ and $a \neq 0$, and choose $N>0$ with the property that
$\phi(a)<N$ for all $a \in D$.  Define $\overline{h} : A \rightarrow
\mathbb N$ by setting $\overline{h}(a)=h(a)$ if $\phi(a)<N$, and
$\overline{h}(a)=0$ otherwise.  Let $D' = D \cup \{a : N \leq \phi(a)
\leq B(N) \}$, where $B(N)$ is as in Lemma~\ref{l:verysupportive}.
By Lemma~\ref{l:verysupportive}, $D'$ is a very supportive set for
$\overline{h}$, so $\Hilb^{\overline{h}}_S \cong
\Hilb^{\overline{h}}_{S_{D'}}$.

Consider now the equations for $\Hilb_{S_D}^{\overline{h}}$ and
$\Hilb_{S_{D'}}^{\overline{h}}$.  The Hilbert scheme $\Hilb_{S_D}$ is
constructed as a subscheme of the product of Grassmannians $\prod_{a
  \in D} G(h(a),S_a)$.  Each ideal $I \in \Hilb_{S_D}$ gives rise to
the codimension-$h(a)$ subspace $I_a$ of $S_a$.  The equations
defining the Hilbert scheme are the quadratic equations in the
Pl\"ucker coordinates on the Grassmannians that record the fact that
for $a, b \in D$, $x^u I_a \subseteq I_b$ for all $x^u \in S_{b-a}$.
See \cite[Corollary 3.15]{HaimanSturmfels} for more details.

  Since $\overline{h}(a)=0$ for all $a
\in D' \setminus D$, the Grassmannian $G(\overline{h}(a),S_a)$ is a
point for all such $a$, so the second Hilbert scheme embeds into the
same product of Grassmannians as the first.  All quadratic equations
in either case then come from pairs $a,b \in D$, so the equations
defining each Hilbert scheme coincide, and
$\Hilb^{\overline{h}}_{S_{D'}} \cong \Hilb^{\overline{h}}_{S_D}$.
Since $\overline{h}(a) = h(a)$ for all $a \in D$, we have
$\Hilb^{\overline{h}}_{S_D} \cong \Hilb^{h}_{S_D}$.  The choice of $D$
being very supportive means that $\Hilb^h_{S_D} \cong \Hilb^h_S$, so
$\Hilb^h_S \cong \Hilb^{\overline{h}}_S$ as required.
\end{proof}

\begin{remark}
Note that Theorem~\ref{t:isom} implies that every pathology that
exists for a positively-graded multigraded Hilbert scheme also exists
for one where the Hilbert function $\overline{h}$ has finite support.
These can be thought of as fixed-loci for group actions on Hilbert
schemes of points in $\mathbb A^n$, so this means that all
(positively-graded) Hilbert schemes are of this form.
In particular, there must exist such Hilbert schemes that are
disconnected (from \cite{Santos}), and that have non-reduced components
(from \cite{Mumford}).
\end{remark}

\begin{corollary} \label{c:reductionstep2}
To decide whether there is an edge in the $T$-graph between a pair of
monomial ideals $M, N \in \Hilb^h_S$, it suffices to assume that $S$
is graded by $\mathbb Z^n/\mathbf{c} \mathbb Z$, and $\sum_{a \in
  \mathbb Z^n/\c \ZZ} h(a) < \infty$.
\end{corollary}

\begin{proof}
By Corollary~\ref{c:reductionstep}, there is an edge between $M$ and
$N$ if and only if there is $\mathbf{c} \in \mathbb Z^n$ and $H \colon
\mathbb Z^n/\mathbb Z \mathbf{c} \rightarrow \mathbb N$ for which $M,N \in \Hc(H)$, and
there is an edge between $M$ and $N$ in $\Hc(H)$.  The resulting
grading by $\mathbb Z^n/\mathbb Z \mathbf{c}$ is positive, so by
Theorem~\ref{t:isom} there is $H'$ with $\sum_{a \in \ZZ^n/\ZZ\c}
H'(a) < \infty$ and $\Hc(H) \cong \Hc(H')$.  Thus there is an edge
between $M$ and $N$ in $\Hc(H)$ if and only if there is an edge
between the ideals corresponding to $M$ and $N$ in $\Hc(H')$ .
\end{proof}

\section{Necessary conditions for a $T$-edge}
\label{s:section2}

In this section we prove Theorem~\ref{thm:main}, giving necessary
conditions for the existence of an edge in the $T$-graph between two
monomial ideals in $\Hc(H)$.  By Corollary~\ref{c:reductionstep2},
this gives a necessary condition for there to be an edge between
monomial ideals in any $\Hilb^h_S$.  The condition that the Hilbert
function $H$ has $\sum_{a \in \mathbb Z^n/\mathbb Z \mathbf{c}} H(a) <
\infty$ is unnecessary in the first part of  this section, so we do not require it.

Recall the definition of an arrow map (Definition~\ref{d:arrowmap}).
We illustrate the concept of an arrow map in the following 
example. 

\begin{example} \label{e:arrowmaps}
\begin{enumerate}
\item Let $S=\K[x,y]$ be graded by $\mathbb Z^2/\mathbb Z(2,-1)$, so
  $\deg(x)=1$, $\deg(y)=2$, and let $\prec$ be the lexicographic order
  with $x \prec y$.  Let $M = \langle x^8,y \rangle$, and $N = \langle
  x^4,y^2 \rangle$.  Then an arrow map between $M$ and $N$ is given by
  the following set of pairs $(m, f(m))$: $\{ (y,x^4), (xy,x^5),$ $
  (x^2y,x^6), (x^3y,x^7)\} \cup \{ (m,m) : m \in \Mon(\langle x^8,
  x^4y, y^2 \rangle)$.  This is illustrated on the left of
  Figure~\ref{f:arrowmap1}.  The grey shaded monomials are the
  standard monomials of $M$, and the monomials encased by the thick
  black line are the standard monomials of $N$.  A dot in the box
  corresponding to a monomial $m$ indicates that $f(m)=m$.

\item Let $S = \K[x,y]$ be graded by $\mathbb Z^2/\mathbb Z(1,-1)$, so
  $\deg(x)=\deg(y)=1$, and let $\prec$ be the lexicographic order with
  $x \prec y$.  Let $M= \langle x^2y^2, xy^3 \rangle$, and $N= \langle
  x^3y,x^2y^2 \rangle$.  Then an arrow map between $M$ and $N$ is
  given by setting $f(x^ay^b)=x^{a+1}y^{b-1}$ for all $x^ay^b \in M$.
  This is illustrated on the right of Figure~\ref{f:arrowmap1}.
\end{enumerate}

\begin{figure}
\center{
\includegraphics[scale=0.8]{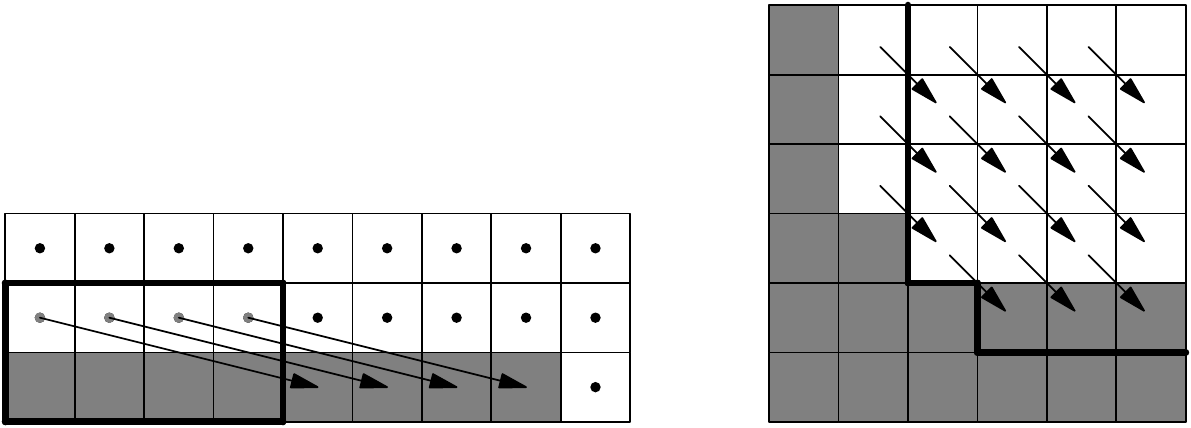}}
\caption{\label{f:arrowmap1} The arrow maps of Example~\ref{e:arrowmaps}. 
}
\end{figure}
\end{example}

\begin{proposition} \label{thm:arrowmapnecessary}
	Let $I$ be homogeneous with respect to a positive
        $\ZZ^n/\ZZ\c$-grading, and let $\prec$ be a term order on
        $S$.  Set $M = \inn_{\prec} I$, and $N= \inn_{\prec^{\opp}}
        I$.  The map $f: \Mon(M) \to \Mon (N)$ defined by
	\[m \mapsto \max_{\prec}\{\inn_{\prec^{\opp}}(p) \mid p \in I \text{ 
homogeneous  
and } \inn_{\prec}(p)=m\}\]
is an arrow map. 
\end{proposition}

\begin{proof}
	First note that the image of $f$ is contained in $\Mon(N)$,
        since $\inn_{\prec^{\opp}}(I) = N$.  Moreover, it follows from
        the definition that $f$ is degree-preserving and $m\succeq
        f(m)$.
For every $m\in \Mon(M)$ choose a homogeneous 
polynomial $p_m\in I$ such that 
$\inn_{\prec}(p_m) = m$ and
$\inn_{\prec^{\opp}}(p_m) = f(m)$.

To see that $f$ is a bijection,  we first show it is injective. Assume
$f(m)=f(m')$ for $m\succ m'$.  We denote by $\lc(f)$ the leading
coefficient of a polynomial $f$.  This is the coefficient of the
largest monomial occurring in $f$ with respect to the term order.
Then $q=\lc_{\prec^{\opp}}(p_{m'})p_m
-\lc_{\prec^{\opp}}(p_m)p_{m'}\in I$ has $\inn_{\prec}q =m$, and
$\inn_{\prec^{\opp}}q \succ f(m)$, which contradicts the construction
of $f(m)$. Since $M$ and $N$ have the same $\mathbb Z^n/\mathbb
Z\mathbf{c}$-graded Hilbert function, and the grading is positive, any
degree preserving injection $\Mon(M) \to \Mon(N)$ is a bijection.

To see condition (\ref{e:mshorter}) of the definition of an arrow map, let $m\in \Mon(M)$ and fix a multiple $m' =
m''m$ of $m$.  Note that $m''p_m \in I$ is homogeneous, with
$\inn_{\prec}(m''p_m)=m'$ and $\inn_{\prec^{\opp}}(m''p_m) = m''f(m)$.
Hence $f(m')\succeq m''f(m)$ and so $d(m',f(m'))\leq d(m,f(m))$.

To show condition (\ref{e:mprimeshorter}) holds, we use the following fact:

\begin{tabular}{cp{11cm}}
($\dag$) & If there exists a homogeneous polynomial $p\in I$ with
  $\inn_{\prec}(p)=m$, $\inn_{\prec^{\opp}}(p)=m'$, then there exists
  $\widetilde{m}\in \Mon(M)$ such that $\widetilde{m} \preceq m$ and
  $f(\widetilde{m})=m'$. In particular $d(\widetilde{m},m') \leq
  d(m,m')$.
\end{tabular}

We prove this fact by induction on $d(m,m')$.  If $d(m,m') = 0$, then
$m=m'=\widetilde{m}=f(m)$, and we can take $\widetilde{m}=m$.
Otherwise, let $p_m \in I $ be a homogeneous polynomial with
$\inn_{\prec}(p_m)=m$ and $\inn_{\prec^{\opp}}(p_m)=f(m)$.  If $f(m) =
m'$, we are done. Assume $f(m) \neq m'$. 
 Then for $q=\lc_{\prec}(p_m)p -\lc_{\prec}(p)p_m \in I$, we
have $\inn_{\prec^{\opp}}(q)=m'$, and $m\succ
\inn_{\prec}(q)=:m_q$.  Since $d(m_q,m') <d (m,m')$, there exists
$\widetilde{m}\preceq m_q \prec m$ with $f(\widetilde{m})=m'$ by the
induction hypothesis, finishing the proof of ($\dag$).

To see condition  (\ref{e:mprimeshorter}), let $m\in \Mon(N)$ and fix a multiple
$m' = mm''$ of $m$. Note that for $p = m''p_{f^{-1}(m)}$, we have
$\inn_{\prec}(p)= f^{-1}(m)m''$ and $\inn_{\prec^{\opp}}(p) =
mm''=m'$.  By ($\dag$) there exists $\widetilde{m}\prec f^{-1}(m)m''$ with
$f(\widetilde{m}) = m$, and (\ref{e:mprimeshorter}) follows.
\end{proof}

The following Corollary is a more general version of 
the first part of Theorem \ref{thm:main}. 

\begin{corollary}\label{cor:main}
	Let $M, N$ be monomial ideals in an arbitrary multigraded 
	Hilbert scheme 
	and assume that there exists an edge 
	between $M$ and $N$ in the $T$-graph. 
	Then there exists an arrow map $f \colon M \to N$ 
        with respect to 
	some  grading by $\ZZ^n/\ZZ\c$ and some term order $\prec$.
\end{corollary}

\begin{example}
Let $\K[x,y]$ have the standard grading, so $\c=(1,-1)$, and let
$\prec$ be the lexicographic term order with $x \prec y$.  Let $I=
\langle x^2 + 2yx + 2y^2, y^4\rangle$.  When $\mathrm{char}(\K) \neq
2$, $\inn_{\prec} I = \langle x^4, y^2 \rangle$ and
$\inn_{\prec^{\opp}} I = \langle x^2, y^4\rangle$.  The arrow map
induced by $I$ is given by $f(y^2) = x^2, f(y^3) = x^2y, f(xy^2) =
x^3, f(xy^3) = x^3y$ and $f(m)=m$ for all other $m \in \langle x^4,y^2
\rangle$.
\end{example}

\begin{example}
	Let $S=\K[x_1,x_2,x_3,x_4]$ be
	graded in by $\mathbb Z^4/\mathbb Z(2,-1,0,0)$, and fix $\prec$ with
	$x_1\succ x_2 \succ x_3 \succ x_4$.  Let $M=\langle
	x_1^2,x_2^2,x_3^2,x_4^2,x_1x_2x_3,x_1x_4,x_2x_4,x_3x_4 \rangle$, and let $N
	= \langle x_2, x_1^4,x_1^3x_3,x_3^2,x_1x_4,x_3x_4,x_4^2 \rangle$.  The
	ideals $M$ and $N$ have  the same 
	$\mathbb Z^4/\mathbb Z(2,-1,0,0)$-graded Hilbert
	function $H$.  The Hilbert scheme $\Hc(H)$ is a subscheme of
	$\Hilb^8(\mathbb A^4)$.  The function $\Mon(M) \rightarrow \Mon(N)$ given
	by $f(x_1^2)=x_2$, $f(x_1^3)=x_1x_2$, $f(x_1^2x_3)=x_2x_3$ and $f(m)=m$
	otherwise is an arrow map.  This arrow map comes from the ideal $I=
	\langle x_2-x_1^2, x_1^4,x_1^3x_3,x_3^2,x_1x_4,x_3x_4,x_4^2 \rangle$.
\end{example}

While Proposition~\ref{thm:arrowmapnecessary} shows that an ideal
gives an arrow map, the following example
shows that not all arrow maps are induced by 
an ideal.

\begin{example}\label{ex:almostall}
Let $S=\K[x,y]$ be graded by $\mathbb Z^2/\mathbb Z(1,-1)$, so
$\deg(x)=\deg(y)=1$, and let $M=\<y^5, x^2\rangle$ and $N = \<y^2,
x^5\rangle$, as in Example~\ref{ex:almostall1}.  
	
Let $\prec$ be the lexicographic order with  $x \prec y$.  Any arrow map $f$ from $\Mon(M)$ to
$\Mon(M')$ must satisfy $f(y^2)=x^2$, $f(y^3)=x^2y$, $f(xy^2)=x^3$,
$f(y^4)=x^2y^2$, $f(xy^3)=x^3y$, and $f(x^2y^2)=x^4$, and  
$f(m)=m$ if $\deg(m) \geq 6$.  However there are three
possibilities for the map $f$ in degree five.  In all cases we have
$f(y^5)=y^5$ and $f(x^5)=x^5$, but we can have $\{f(xy^4)=x^2y^3,
f(x^2y^3)= x^4y, f(x^3y^2)= x^3y^2\}$, $\{f(xy^4)= x^3y^2, f(x^2y^3)=
x^2y^3, f(x^3y^2)=x^4y \}$, or $\{f(xy^4)=x^2y^3 , f(x^2y^3)= x^3y^2,
f(x^3y^2)= x^4y\}$.  Of these, only the last one 
is induced by an ideal as in the statement of 
Proposition~\ref{thm:arrowmapnecessary}.
Indeed, for any ideal $I \in E(M,N)$ we have $axy^4 + bx^2y^3 \in I$, and the equations for $E(M,N)$ imply that 
$b\neq 0$
and $a\neq 0$. So $f(xy^4)=x^2y^3$ for any arrow map induced from $I
\in E(M,N)$.  The analogous equation $cx^4y+dx^3y^2 \in I$ rules out
the second possibility.
\end{example}

\begin{proof}[Proof of Theorem~\ref{thm:main}]
	  The first part of the theorem is a special case of
        Corollary \ref{cor:main}.  For the second, we first observe that for $I \in
        \Hc(H)$ we have $Q \subseteq I$.  Indeed, for every $i$ choose
        whichever term order $\prec$ or $\prec^{\opp}$ agrees with the
        lexicographic term order with $x_i$ smallest. Since $x_i^{r_i}
        \in M \cap N$, there exists $f$ with initial term $x_i^{r_i}$
        with respect to this term order, so $x_i^{r_i} \in I$.

 To see the existence of an arrow map $\hat{f} \colon \Mon((Q:M))
 \rightarrow \Mon(Q:N))$, it suffices to show that for 
any term order $\prec$ we have
\[ \inn (Q:I) = (Q: \inn (I)).\]
This means that $(Q:I)$ is homogeneous with respect to the $\ZZ^n/\ZZ\c$-grading with two initial ideals $\inn_{\prec} (Q:I) = (Q:M)$ 
and $\inn_{\prec^{\opp}} (Q:I) = (Q:N)$. The existence of 
an arrow map now follows from Proposition~\ref{thm:arrowmapnecessary}.

For any ideals $J,K$, we have $\inn(J)\inn(K) \subseteq \inn(JK)$, so 
$\inn (Q:I) \inn(I) \subseteq \inn((Q:I)I) \subseteq \inn Q = Q$.
This implies that $\inn (Q:I) \subseteq (Q: \inn(I))$.  Since
$\dim_{\K}(S/\inn(Q:I)) =\dim_{\K}S/(Q:I)$, to show equality it
suffices to show that $\dim_{\K}(S/(Q:I)) = \dim_{\K}(S/(Q: \inn(I)))$.  

Note that $S/Q$ is a zero-dimensional ring that is a complete
intersection, hence Gorenstein.  Thus $D(-)=\Hom_{S/Q}(-,S/Q)$ is a
dualizing functor from the category of finitely generated
$S/Q$-modules to itself (see \cite[\S 21.1 and \S 21.2]{Eisenbud}).
Since $Q\subset I$, $S/I$ is a $S/Q$-module.  We have an isomorphism
$D(S/I) = \Hom_{S/Q}(S/I,S/Q) \cong (0:_{S/Q} I/Q) =
(Q:I)/Q$, where the isomorphism takes $\phi \in \Hom_{S/Q}(S/I,S/Q)$
to $\phi(1)$. Thus 
$\dim_{\K}(S/I) = \dim_{\K}(D(S/I)) = \dim_{\K}((Q:I)/Q) =
\dim_{\K}(S/Q) - \dim_{\K}(S/(Q:I))$.  The desired equality follows
from the fact that $\dim_{\K}(S/I) = \dim_{\K}(S/\inn(I))$.
\end{proof}

The following example shows that the conditions of
Theorem~\ref{thm:main} (\ref{thm:main1}) and (\ref{thm:main2}) are not
equivalent. In particular, for monomial ideals $M$ and $N$, the
existence of an arrow map $f\colon M \to N$ is not sufficient for the
edge $E(M,N)$ to be nonempty.  We do not have an example where both
conditions of Theorem~\ref{thm:main} are not sufficient.

\begin{example} 
	\label{ex:all}
Let  $S=\K[x,y]$ be graded by 
$\deg(x)=\deg(y)=1$ and let $\prec$ be the lexicographic order with $x \prec y$. 
Let $M
= \langle x^5, x^3y^2,y^4\rangle$, $N
=\langle x^4, x^3y^3,xy^4,y^5\rangle$, and  
	 let $Q = \langle x^5,y^5\rangle$.  This is illustrated in 
 Figure \ref{f:nineexample}. 
Then the map $f\colon \Mon(M) \to \Mon(N)$ defined by 
$f(y^4)= x^4$, $f(x^3y^2)=x^4y$, and $f(m)=m$ otherwise is 
an arrow map.  

However,  $(Q:M)=\langle x^5, x^2y, y^3 \rangle$ and
$(Q:N)=\langle x^4, x^2y, xy^2, y^5 \rangle$.  There is no arrow map
$g\colon (Q:M)\to (Q:N)$.  If there were an arrow map $g$, by
$\eqref{e:bijectivedecreasing}$ we would have $g(y^3)=xy^2$ and
$g(x^2y)=x^2y$. Then $\eqref{e:mshorter}$ applied to $y^3$ implies
that $g(y^4)=xy^3$ and $g(xy^3)=x^2y^2$, a contradiction to
$\eqref{e:mprimeshorter}$ applied to $x^2y$.  Note however that the map
given by $g(y^3)=xy^2, g(x^2y)=x^2y, g(y^4) = xy^3, g(xy^3) = x^4$ is
a \emph{system of arrows} in the sense of \cite{EvainSchubert}; see
Remark~\ref{r:incidence}.  Compare \cite[Section 4]{Yameogo1},
\cite[Section 5]{EvainSchubert}.
\begin{figure}
	\caption{\label{f:nineexample} The arrow maps and systems of arrows of Example \ref{ex:all}.}
		\includegraphics[scale=1]{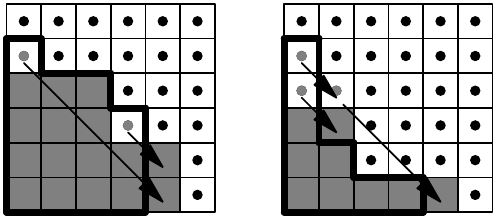}
	\end{figure}
\end{example}

\section{The Hilbert scheme of points in the plane}

In this section we discuss in more detail the case of the Hilbert
scheme $\Hilb^d(\mathbb A^2)$ of $d$ points in the plane.  In this
case the vertices of the $T$-graph correspond to partitions of $d$.

\subsection{The basic structure of the $T$-graph of $\Hilb^d(\AA^2)$.}
As explained in Corollary~\ref{c:reductionstep} and Remark~\ref{r:tangentweights}, the $T$-graph of
$\Hilb^{d}(\AA^2)$ decomposes as a union of $T$-graphs of finitely
many different $\Hc(H)$ where $\c \in \ZZ^2$ with $\mathbf{c}^+,
\mathbf{c}^- \neq 0$, and $H \colon \ZZ^2/\mathbb Z \c \to \mathbb N$
is a Hilbert function.  In this situation, $\Hc(H)$ is smooth and
irreducible; see \cite{EvainIrreducible, Iversen72,
  MaclaganSmithHilbert}.  This is not true when $S$ has more than two
variables.

Up to sign we have ${\c} = (\beta,-\alpha)$, where $\alpha, \beta \in
\mathbb Z_{>0}$ are relatively prime.  Thus $S=\K[x,y]$ has a $\mathbb
Z$-grading by $\deg(x)=\alpha, \deg(y) = \beta$, and $\Hc(H)$
consists of all ideals that are homogeneous with respect to this
positive grading and that have Hilbert function $H$.

   Recall from Definition~\ref{d:EMN} that for monomial ideals $M,N$
   in $\Hc(H)$ the edge-scheme $E(M,N)$ of one-dimensional torus
   orbits connecting $M$ and $N$ is given by $C_{\prec}(M) \cap
   C_{\prec^{\opp}}(N)$.  Note that $E(M,N)$ is is empty unless $M>N$
   in the partial order of Definition~\ref{d:posetorder}.

For the remainder of this section, we let  
$\prec$ denote the lexicographic  term order 
with $x\prec y$ and 
$\prec^{\opp}$ denote the lexicographic term order 
with $y\prec x$.

\begin{definition}\label{d:significant}
Fix a $\mathbb Z^2/\mathbb Z \mathbf{c}$-grading,  
and let $M\subset \K[x,y]$ be a monomial ideal of finite length. 
Let 
$m_0 = x^{a_0}\prec m_1=x^{a_1}y^{b_1} \cdots \prec m_e 
=y^{b_e}$ be the minimal generators of $M$ and let
$r=x^{\beta}y^{-\alpha}.$

Let $w_i = \lcm(m_{i-1},m_{i})$.  A \emph{positive significant arrow}
is a pair $c_i^{\ell} = (m_i, \ell)$, where $\ell \in \ZZ_{>0}$, such
that the monomial $m_ir^{\ell}$ is a monomial not in $M$, and such that $w_ir^{\ell} \in M$.  A
\emph{negative significant arrow} 
is a pair $c_i^{\ell} = (m_i,
\ell)$, where $\ell \in \ZZ_{<0}$, such that the monomial $m_{i+1}r^{\ell}$ is a monomial not in $M$ and 
$w_{i+1}r^{\ell} \in M$.  We denote by
$T_+(M)$ the set of positive significant arrows, and by $T_-(M)$ the
set of negative significant arrows.  
Note that $w_i = x^{a_{i-1}}y^{b_i}$, and the condition that $m_ir^l$ is a
monomial means that $\ell \alpha \leq b_i$.
\end{definition}

\begin{definition} \label{d:IM}
To every monomial $m\in M$  we associate a 
minimal generator $m_{j(m)}$ of $M$, where 
\begin{equation} \label{eq:jdefn}
j(m)  = \max\{ j \mid m_j 
\text{ divides } m \}.\end{equation}
Let $f_0 = m_0$ and define recursively 
\[f_{i} = \frac{m_{i}}{m_{i-1}}f_{i-1} + \sum_{c_{i}^{\ell}
\in T_+(M)} c_{i}^{\ell}\frac{m_{i}r^{\ell}}
   {m_{j(w_ir^{\ell})}}f_{j(w_ir^{\ell})} \in \K[T_+(M)][x,y] ,\]
   where we abuse notation by identifying the significant arrow
   $c_i^{\ell}$ with the corresponding variable.  That $f_i$ is a
   polynomial (as opposed to merely a Laurent polynomial) follows by
   induction.
  Let $$I_{\prec}(M) =
   \langle f_0,\ldots, f_e\rangle \subset \K[T_+(M)][x,y].$$
\end{definition}

\begin{example}\label{ex:runningex}
	Let $\K[x,y]$ have the standard grading $\deg(x) =
        \deg(y) = 1$. Then $r=xy^{-1}$.  Let $M =
        \langle x^8, x^5y, x^3y^3,y^4\rangle$.  Then $m_0 = x^8, m_1 =
        x^5y, m_2 =x^3y^3$ and $m_3 = y^4$, and we have $w_1 = x^8y,
        w_2 = x^5y^3$, and $w_3= x^3y^4$.  The positive significant arrows are
        $T_+(M) = \{c_1^1, c_2^1, c_2^3, c_3^1, c_3^2, c_3^3\}$.  This
        is illustrated in Figure~\ref{f:runningeg}.  The polynomials
        are
	\begin{eqnarray*}
		f_0 &=& x^8 \\
		f_1 &=& x^5y + c_1^1x^6\\
		f_2 &=& x^3y^3+(c_1^1+c_2^1)x^4y^2 +
		(c_2^1c_1^1)x^5y + c_2^3x^6 \\
		f_3 &=& y^4+(c_1^1+c_2^1+c_3^1)xy^3 +
		(c_2^1c_1^1+c_3^1c_1^1+c_3^1c_2^1+c_3^2)x^2y^2\\
		&&+(c_2^3+c_3^1c_2^1c_1^1+c_3^2c_1^1+c_3^3)x^3y
		+(c_3^1c_2^3+c_3^3c_1^1)x^4.\\
	\end{eqnarray*}
\begin{figure}
\caption{\label{f:runningeg}
        The positive significant arrows for the monomial ideal
	$\langle  x^8, x^5y, x^3y^3,y^4\rangle$ of 
	Example \ref{ex:runningex}. 
	Here $c_1^1$ is green, 
	$c_2^1, c_2^3$ are blue, and $c_3^1, c_3^2, c_3^3$ are 
	red. }
\includegraphics{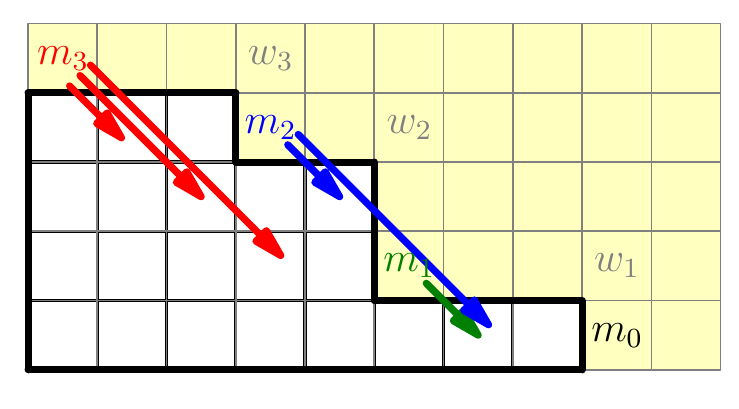}
\end{figure}
\end{example}

\begin{remark} \label{r:sigarrows}
The set of
significant arrows of $M$, for all possible (not necessarily positive)
gradings $\ZZ^2/\ZZ\c$, is in bijection with the weights of the torus
action on the tangent space $T_M\Hilb^d(\AA^2)$ to $M$ induced by the
action of $T$ on $\Hilb^d(\AA^2)$. 
The map that associates to a
significant arrow $c_i^{\ell}$ for the grading $\ZZ/\ZZ\c$ the weight
$\ell\c$ is a bijection.  
See for example \cite[Proposition 5.7]{Nakajimabook},
\cite[Proposition 2.4]{HaimanQTCatalan}, or \cite[\S
  2]{EvainIrreducible}.
	In particular, the gradings by $\ZZ^n/\ZZ\c$  
	for which $\Hc(H)$  is not simply the point $\{M\}$,
	where $H$ is the Hilbert function of $M$, are those 
	for which there exists some 
        significant arrow for $M$; compare Remark 
	\ref{r:tangentweights}.
\end{remark}

In \cite[Section 3]{EvainIrreducible}, Evain gives the 
following 
parametrization of the cells $C_{\prec}(M)$ which
we will use to compute equations for the edge-scheme $E(M,N)$.

\begin{theorem}\label{t:Evain}
We have
	\begin{enumerate}
\item   The set $\{f_0, \ldots, 
	f_e\}$ is a Gr{\"o}bner basis of $I_{\prec}(M)$ with respect to $\prec$  
(\cite[Proposition 10]{EvainIrreducible}).
\item \label{t:Evain2} The map $\mathbb A^{|T_+(M)|} \to \Hc(H)$ induced by 
	$I_{\prec}(M)$ is an isomorphism onto the affine cell $C_{\prec}(M)$
(\cite[Theorem 11]{EvainIrreducible}).
\end{enumerate}
\end{theorem}

In the following proposition we use 
arguments from Evain \cite{EvainIrreducible} to imply that $\Hc(H)$ is
the closure of an edge-scheme.

\begin{proposition} \label{t:Hab}
	Let $S = \K[x,y]$ be graded by $\ZZ^2/ \ZZ\c$, where $\c \in
        \ZZ^2$ with $\c^+, \c^-\neq 0$.  Fix  
        a Hilbert function $H$.  Then there exists a unique maximal
        element $M_{\max}$ and a unique minimal element $M_{\min}$
        with respect to the partial order of
        Definition~\ref{d:posetorder} for the monomial ideals
        contained in $\Hc(H)$, and $\Hc(H)$ is the closure of the edge-scheme 
        $E(M_{\max},M_{\min})$.
\end{proposition}

\begin{proof}
Evain \cite[Theorem 19]{EvainIrreducible} shows 
that the poset of
Definition~\ref{d:posetorder} on the monomial ideals contained in
$\Hc(H)$ has a unique minimal element $M_{\min}$ such that
$T_+(M_{{\min}}) = \emptyset$.  By Remark~\ref{r:sigarrows}, the
number of significant arrows at $M$ equals the dimension of the
tangent space to $\Hc(H)$ at $M$, which equals the dimension since
$\Hc(H)$ is smooth. So since $T_+(M_{{\min}})= \emptyset$, we have
$|T_-(M_{\min})|=\dim \Hc(H)$. 

By switching $x$ and $y$ we switch the roles of $\prec$ and
$\prec^{\opp}$, and of positive and negative signficant arrows.  Thus
Theorem~\ref{t:Evain}(\ref{t:Evain2}) also applies to $\prec^{\opp}$,
and so we have $\dim C_{\prec^{\opp}}(M_{\min}) = \dim \Hc(H)$.

The cells $C_{\prec}(M)$ are locally closed, so
$C_{\prec^{\opp}}(M_{\min})$ is open; see \cite[Theorem 4.2]{BCM}.
  Since $\Hc(H)$ is irreducible \cite{EvainIrreducible}, it follows
  that $C_{\prec^{\opp}}(M_{\min})$ is an open dense subset of
  $\Hc(H)$.  Similarly, we obtain a
  maximal element $M_{\max}$ such that $C_{\prec}(M_{\max})$ is an
  open dense subset of $\Hc(H)$. Hence the intersection
  $C_{\prec}(M_{\max}) \cap C_{\prec^{\opp}}(M_{\min})$ is an open
  dense subset of $\Hc(H)$, which implies the claim.
\end{proof}
\begin{remark}
Our proof uses the smoothness of $\Hc(H)$ in
this setting, and it would be interesting to know if this closure
property is true in more than two variables, where the smoothness may
fail.  It also suggests studying the restricted graph of just edges
whose closure is the entire $\Hc(H)$.
\end{remark}

\begin{remark}
	In the standard grading $\deg(x) = \deg(y)=1$, Macaulay's
        theorem asserts that the maximal element of the poset is the
        lex-segment ideal with Hilbert function $H$, which is the
        monomial ideal containing the $d+1-H(d)$ largest elements of
        degree $d$ for every $d$; see, for example,  \cite[Theorem
          4.2.10]{BrunsHerzog}.  This is not true for more general
        gradings.  However, Evain gives a recursive construction of
        $M$ in \cite[Remark 23]{EvainIrreducible}; see also
        \cite[Proposition 3.12]{MaclaganSmithHilbert}.
\end{remark}

\begin{remark}\label{r:incidence}
Yam\'eogo \cite{Yameogo1, Yameogo2} and Evain \cite{EvainSchubert}
studied a related incidence question in the case of
$\Hilb^{d}(\AA^2)$. In \cite{Yameogo1} Yam\'eogo shows that the
closure of a cell $C_{\prec}(M)$ need not be a union of cells.  Given
two cells $C_{\prec}(M)$, $C_{\prec}(M')$, one may ask whether
$C_{\prec}(M') \subset \overline{C_{\prec}(M)}$ (strong incidence) or
$C_{\prec}(M') \cap \overline{C_{\prec}(M)} \neq \emptyset$ (weak
incidence), and in \cite{Yameogo2} Yam\'eogo gives a sufficient
condition for strong incidence.  In \cite{Yameogo1} he shows that
being related in the partial order of Definition~\ref{d:posetorder} is
a necessary, yet not sufficient, condition for weak incidence.
Evain strengthens this condition in \cite{EvainSchubert} as follows.
A \emph{system of arrows} is a monomial map $f:\Mon(M) \to \Mon(N)$
satisfying conditions \ref{e:bijectivedecreasing} and
\ref{e:mprimeshorter} in Definition~\ref{d:arrowmap}; compare
\cite[Definition-Proposition 11]{EvainSchubert}.  Evain's systems of
arrows \cite[Definition 4]{EvainSchubert} are defined to be maps on
the partitions, so they go from monomials not in $M$ to monomials not
in $N$. However, Evain's system of arrows is  equivalent to a 
system of arrows in this
sense from $\Mon(Q \colon M)$ to $\Mon(Q \colon N)$.

In \cite[Theorem
  8]{EvainSchubert}, Evain proves that if $C_{\prec}(N)\cap
\overline{C_{\prec} (M)}\neq \emptyset$ and $Q$ is as in
Theorem~\ref{thm:main}, then there exists a system of arrows $\Mon(M)
\to \Mon(N)$ and a system of arrows $\Mon(Q\colon M) \to \Mon(Q\colon
N)$.  No example is known where this condition is not
sufficient.  Note that in Example \ref{ex:all}, the map given by
$g(y^3)=xy^2, g(x^2y)=x^2y, g(y^4) = xy^3, g(xy^3) = x^4$ is a system
of arrows between $(Q:M)$ and $(Q:N)$, but there is no arrow map.

The existence of a one-dimensional torus orbit between $M$
and $N$ implies that $N \in \overline{C_{\prec}(M)}$ and $M \in
\overline{C_{\prec^{\opp}}(N)}$, so Evain's theorem \cite[Theorem
  8]{EvainSchubert} implies the existence of a system of arrows from
$M \to N$ and $(Q\colon M) \to (Q\colon N)$ with respect to $\prec$ as
well as $\prec^{\opp}$. 
\end{remark}
\subsection{A combinatorial description of the edge ideal}
\label{s:utahequations}

In this section we prove Theorem~\ref{t:combinatorial}, by giving an
explicit combinatorial description of the equations for the edge
scheme $E(M,N)$ over $\mathbb Z$.  The algorithm in \cite[Algorithm
  5]{AltmannSturmfels} for computing this edge works by
finding Gr{\"o}bner bases for $I_{\prec}(M)$ and $I_{\prec^{\opp}}(N)$
and reducing the Gr{\"o}bner basis for $I_{\prec^{\opp}}(N)$ modulo
the Gr{\"o}bner basis $I_{\prec}(M)$.  We now apply this algorithm to
Evain's Gr{\"o}bner basis to combinatorially describe equations for
$E(M,N)$, which involves the following combinatorial constructions.

	 For this section we fix  a $\mathbb Z^2/\mathbb Z \mathbf{c}$-grading,  
	  and a monomial ideal $M \subset \K[x,y]$.   
The following
definition, which uses the notation of Definition \ref{d:significant}, defines the
combinatorial objects we will use to give equations for the $E(M,N)$.
\begin{definition}
	\begin{enumerate}
		\item 
			A \emph{path} 
	 from a generator $m_i$ of $M$ 
	 is a sequence of  arrows 
	 $\cP = (c_{i_1}^{\ell_1}, \ldots, c_{i_d}^{\ell_d})$
	 where $c_{i_k}^{\ell_k}\in T_+(M)$, 
	 defined inductively as follows:

	 \begin{enumerate}
		 \item If $c_i^{\ell} \in T_+(M)$,
			 then $(c_i^{\ell})$ is a path 
			 from $m_i$. 
		 \item Otherwise, either 
			 \begin{itemize}	
		 \item[$(\dag)$] $(c_{i_1}^{\ell_1}, \ldots,
                   c_{i_d}^{\ell_d})$ is a path from $m_{i-1}$, or
				 \item[$(\ddag)$] $i_1 = i$,
                                   and $(c_{i_2}^{\ell_2}, \ldots,
                                   c_{i_d}^{\ell_d})$ is a path from
                                   $m_{j(w_1r^ \ell)}$.
			 \end{itemize}
			 \end{enumerate}

We define the \emph{length} of the path $\cP$ to be $\ell(\cP) =
\ell_1+ \cdots + \ell_d$, and we say that $\cP$ is a path from from
$m_i$ to $m_i r^{\ell(\cP)}$.  The construction of a path guarantees
that $m_ir^{\ell(\cP)}$ is a monomial; when $d=1$ this is part of the
definition of a significant arrow.

		 \item 
	A \emph{walk} from a generator $m_i$ of $M$ to a monomial $s$
        is defined to be a sequence of paths $(\cP_1, \ldots, \cP_d)$,
        such that $\cP_1$ is a path from $m_i$,
        $\cP_{k}$ is a path from $m_{j\left(m_ir^{\ell(\cP_1)+ \cdots
            \ell(\cP_{k-1})}\right)}$ for $2\leq k\leq
        d$, 
        $m_ir^{\ell(\cP_1)+\cdots + \ell(\cP_k)} \in M$ for $1\leq
        k\leq d-1$, 
and $m_ir^{\ell(\cP_1)+\cdots + \ell(\cP_d)} =s$.  We define the \emph{length} of the walk to be $\ell(\cW)
        = \ell(\cP_1) + \cdots + \ell(\cP_d)$.

		\item 
	A \emph{stroll} from a monomial $m \in M$ to a standard
        monomial $s$ is a sequence of walks $\cS = (\cW_1, \ldots,
	\cW_d)$, such that $\cW_1$ is a walk from $m_{j(m)}$,
$\cW_k$ is a walk
        from $m_{j\left(mr^{\ell\left(\cW_1\right) + \cdots +
            \ell\left(\cW_{k-1}\right)}\right)}$ for $2\leq k\leq d$, 
        $mr^{\ell\left(\cW_1\right) + \cdots +
          \ell\left(\cW_{k-1}\right)} \in M$, and 
        $mr^{\ell\left(\cW_1\right) + \cdots +
          \ell\left(\cW_{d}\right)}=s$.
        There is also the trivial stroll from a standard monomial $s$
        for $M$ to itself.
\item 
\label{d:hike}
	Let $H$ be a Hilbert function, and let $M, N$ be monomial
        ideals in $\Hc(H)$ such that $M>N$ in the partial order of
        Definition~\ref{d:posetorder}.  A \emph{hike} $\cH$ from a
        minimal generator $n$ of $N$ to a standard monomial $s$ for
        $M$ is a pair $(\cP', \cS)$, where $\cP'$ is  a path for $N$ with
        respect to $\prec^{\opp}$ from $n$ to some monomial $m\succ n$
        or $\cP'$ is an  arrow of length zero, in which case we set $m=n$, and
        $\cS$ is a stroll from $m$ to $s$.
	\end{enumerate}

For a path $\cP$, we set $a_{\cP} = \prod _{c_i^{\ell} \in \cP}
c_i^{\ell} \in \K[T_+(M)]$.  For a
walk $\cW$, we set $a_{\cW} = (-1)^{|\cW|+1}\prod_{\cP\in \cW} a_{\cP}$. For a
stroll $\cS$ we let $a_{\cS} = \prod _{\cW\in \cS}(-1)^{|\cS|} a_{\cW}
$  and for the trivial stroll $\cS$ we let $a_{\cS} = 1$.
For a hike $\cH = (\cP', \cS)$, we let $a_{\cH} =
a_{\cP'}a_{\cS}$, and if $\cP'$ is the arrow of length zero we set $a_{\cP'}= 1$. Note that then  $a_{\cP}, a_{\cW}, a_{\cS} \in
\K[T_+(M)], $ while $a_{\cH}\in
\K[T_+(M),T_-(N)]$.
\end{definition}

We will see in the lemmas below that the notion of a path naturally
comes from the recursive definition of the $f_i$ from
Definition~\ref{d:IM}, the notion of a walk from computing the reduced
Gr{\"o}bner basis for $I_{\prec}(M)$, and the notion of a stroll from
reducing a monomial with respect to the reduced Gr{\"o}bner basis.

\begin{example}
	\label{ex:thebigexample}
\begin{figure}
\caption{\label{f:pathsetc}
The upper left picture shows the path $c_2^1c_1^1$ from 
$m_2$ and the paths $c_3^3c_1^1$ and $c_3^1c_2^1c_1^1$ from 
$m_3$; the upper right picture shows  
the walk $c_2^1c_1^1c_1^1$ from $m_2$; the lower left the stroll
$c_3^1c_2^1c_2^1c_1^1c_1^1$ from $xy^5$; the lower right the 
hike $\widetilde{c}_1^4 c_3^1c_2^1c_2^1c_1^1c_1^1$ from 
$x^5y$ to $x^6$.}
\includegraphics[scale=0.8]{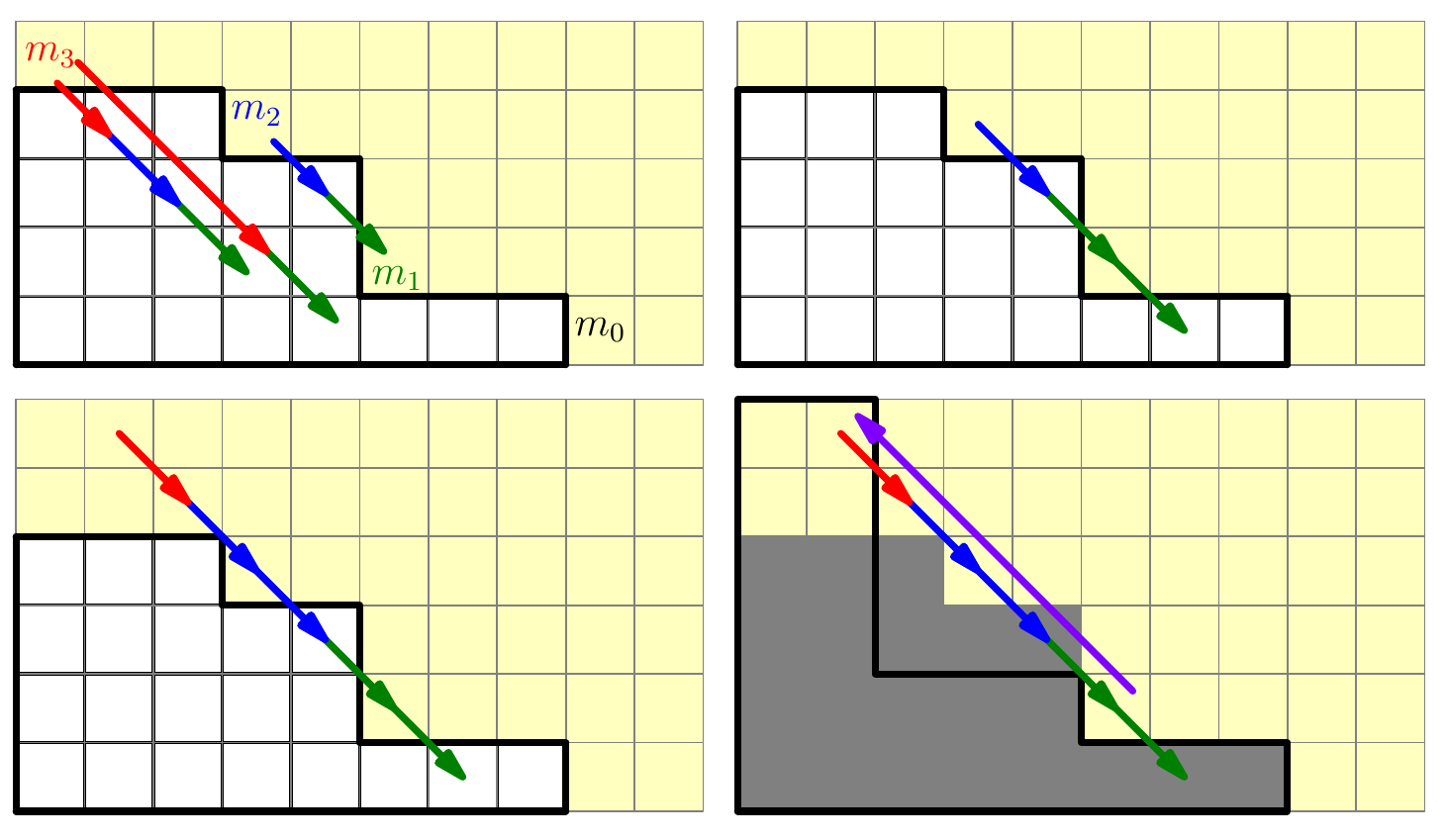}

\end{figure}
	We illustrate the concepts of paths, walks, and strolls
        on the ideal 
       $M= \langle x^8, x^5y, x^3y^3,y^4\rangle$
of Example \ref{ex:runningex}.  This is shown in
        Figure \ref{f:pathsetc}.  It is convenient to describe paths,
        walks, strolls, and hikes as terms of polynomials in
	$\K[T_+(M)]$, where the path $\cP$ is represented by the term 
	$a_{\cP}$. There are no paths from $m_0$.  The only path
        from $m_1$ is $c_1^1$, the paths from $m_2$ are the terms of $
        c_2^1+c_1^1+c_2^1c_1^1+c_2^3$, and the paths from $m_3$ are
        the terms of $ c_3^1+ c_1^1+c_2^1+
        c_3^2+c_2^1c_1^1+c_3^1c_1^1+c_3^1c_2^1+
        c_3^3+c_2^3+c_3^1c_2^1c_1^1+c_3^2c_1^1+ c_3^1c_2^3+c_3^3c_1^1
        $. 
        For $i=0,1,3$, a walk from $m_i$ is simply a
        path from $m_i$.  The walks from $m_2$ are the paths from
        $m_2$ as well as $(c_2^1c_1^1,c_1^1)$.

        The only  stroll from 
	$x^5y$ to $x^6$ is $c_1^1$. There is only the 
	trivial stroll from $x^4y^2$. The strolls from 
	$x^3y^3$ to $x^6$ are 
	$c_2^3$, and $c_2^1c_1^1c_1^1$. 
	For the more complicated strolls it is convenient 
	to define $W_m^{\ell}$ to be the polynomial 
	whose terms are walks from $j(m)$ of length $\ell$
	(the $\ell$-homogeneous part of the walk 
	polynomials computed above),
	and to define $S_{m,s}$ to be the polynomial whose terms are strolls 
	from $m$ to $s$. 
	Then $S_{x^2y^4,x^6} = W_{x^2y^4}^{4}+W_{x^2y^4}^{3}
	S_{x^5y,x^6} + W_{x^2y^4}^{1}S_{x^3y^3,x^6} 
=2c_3^1c_2^3+c_3^3c_1^1+2c_2^3c_1^1+2c_3^1c_2^1(c_1^1)^2+c_3^2(c_1^1)^2+c_2^1c_2^3+c_2^1(c_1^1)^3+(c_2^1)^2(c_1^1)^2	
	$
	and 
	\begin{eqnarray*}
	S_{xy^5,x^6} &=& W_{xy^5}^{4}S_{x^5y,x^6}+
	W_{xy^5}^{2}S_{x^3y^3,x^6}+W_{xy^5}^{1}S_{x^2y^4,x^6} \\
	&=& 6c_3^1c_2^3c_1^1+c_3^2c_2^3+3c_2^3c_2^1c_1^1+
	4c_3^1c_2^1c_2^3+2c_3^2c_2^1(c_1^1)^2\\&&+3(c_2^1)^2(c_1^1)^3
+	4c_3^1c_2^1(c_1^1)^3+2c_3^1(c_2^1)^2(c_1^1)^2+2(c_3^1)^2c_2^3\\
&&+c_3^1c_3^3c_1^1+2(c_3^1)^2c_2^1(c_1^1)^2+c_3^1c_3^2(c_1^1)^2+
2c_2^1c_3^3c_1^1\\&&+2c_3^1(c_2^1)^2(c_1^1)^2+(c_2^1)^2c_2^3+(c_2^1)^3(c_1^1)^2+c_3^3(c_1^1)^2\\&&+2c_2^3(c_1^1)^2+c_3^2(c_1^1)^3+c_2^1(c_1^1)^4.
\end{eqnarray*} 
\end{example}
        Let \[\cM_i = \{m \text{ a monomial } \mid m\prec m_i \text{ and }
	\deg(m) = \deg(m_i)\},\] 
where $m_i$ is as in Definition~\ref{d:significant}.
\begin{lemma}
	\label{prop:evainbasis}
        We have 
	\[ f_i =  m_i + \sum_{m\in \cM_i}\left( \sum_{\substack{\cP
	\text{ path from } \\m_i \text{ to }m} 
	} a_{\cP}\right) m.\]
\end{lemma}

\begin{proof}
	We proceed by induction on $i$. For $i=0$, we have $\cM_i =
        \emptyset$ and $f_0 = m_0$.  The coefficient of $m$ in $f_{i}$
        is the coefficient of $\frac{m_{i-1}}{m_{i}}m$ in $f_{i-1}$
        plus the sum over all significant positive arrows
        $c_{i}^{\ell}$ originating at $m_i$ of $c_{i}^{\ell}$ times
        the coefficient of $\frac{m_{j(w_ir^{\ell})}}{m_{i}r^{\ell}}m$
        in $f_{j(w_ir^{\ell})}$.  By induction, the former corresponds
        to the paths of the form b) $(\dag)$, and the latter to the paths
        of the form b) $(\ddag)$ and a).
\end{proof}

\begin{lemma} \label{l:reducedbasis}
	Let $\cT_i = \{ s \in \cM_i \mid  s\notin M\}$ and 
       \[g_i = m_i + \sum_{s\in \cT_i}\ \left(\sum_{\substack{
       \cW \text{ walk from} \\ m_i \text{ to }s}} a_{\cW}\right) s.\] Then
       $\{g_0, \ldots, g_e\}$ is a reduced Gr{\"o}bner basis for the
       ideal $I_{\prec}(M)$ of Definition~\ref{d:IM}.

\end{lemma}

\begin{proof}
        We first prove that each $g_i$ lies in the ideal $\langle f_0,
        \ldots, f_e\rangle$.

	Let $\cM_{i} \smallsetminus \cT_{i} = \{u_1\ldots, u_t\}$,
        with $ u_1\succ \cdots \succ u_t$.  We let
        $U_k=\{u_1, \ldots, u_k\}$ and $U_0 = \emptyset$. 
	A walk $\cW = (\cP_1, \ldots ,
        \cP_d)$ is a called a \emph{$k$-walk} if
        $m_ir^{\ell(\cP_1)+\cdots+\ell(\cP_j)} \in U_k $ for $1\leq
        j\leq d-1$.  Note that we do not require $m_ir^{\ell(\cW)}$ to
        be contained in $U_k$.  A $0$-walk is a path, and a $t$-walk
        is a walk.

For  $0\leq k \leq t$ we let 
	\[ 
	h_k = m_i + \sum_{m\in \cM_i \smallsetminus U_k}
         \left( 	\sum_{\substack{\cW \text{ k-walk from } \\m_i \text{ to } 
	m}} a_{\cW}
	\right) m,\] 
	where the inner sum is over all 
	$k$-walks from $m_i$ to $m$. Note that 
	$h_0 = f_i$ and $h_t = g_i$.

	Now for $k\geq 1$, we have  
	\begin{IEEEeqnarray}{rCl}
		\IEEEeqnarraymulticol{3}{l}{
		h_{k-1} - f_{j(u_{k})}\frac{u_{k}}{m_{j(u_{k})}}
		\left( \sum_{\substack{\cW \text{ $(k-1)$-walk} \\ \text{from }
		m_i \text{ to } u_{k}}} a_{\cW} 
	\right)} \nonumber\\
	&=&  m_i+\sum_{m\in \cM_i \smallsetminus U_{k-1}}
	\left(\sum_{\substack{\cW \text{ $(k-1)$-walk} \\ \text{from } m_i \text{ to } 
	m}}  
	a_{\cW}\right) m \nonumber \\
&&	-\left( m_{j(u_k)} + 
	\sum_{\widetilde{m}\in \cM_{j(u_{k})}} 
	\left(\sum_{\substack{\cP \text{ path from}\\ m_{j(u_{k})} \text{ to } \widetilde{m}}} a_{\cP} 	\right)\widetilde{m}\right) \frac{u_{k}}{m_{j(u_{k})} }
	\left( \sum_{\substack{\cW \text{ $(k-1)$-walk} \\\text{ from } m_i \text{ to } u_{k}}}   a_{\cW} \right)\nonumber \\
&	=&  
	m_i+ \sum_{m\in \cM_i \smallsetminus U_{k}}
	\left[
	\sum_{\substack{\cW \text{ $(k-1)$-walk} \\ \text{from } m_i \text{ to } m}}  
	a_{\cW}  - 
	\left(
	\sum
	_{\substack{\cP \text{ path from}\\ m_{j(u_{k})} \text{ to }\\\frac{ mm_{j(u_k)}}{u_k}}}
a_{\cP} 	
	\right)  
	\left( \sum
	_{\substack{\cW \text{ $(k-1)$-walk} \\\text{ from } m_i \text{ to } u_{k}}}
	a_{\cW} \right)
	\right] m \nonumber \\
& = & h_k, \nonumber
	\end{IEEEeqnarray}
where the second-to-last equality follows from the fact that the
coefficient of $u_{k}$ cancels, and the last follows from the
definition of a walk.  It follows that $g_i = h_t \in \langle f_0,
\cdots, f_e \rangle$.  Since 
$\langle \inn_{\prec}(g_0), \ldots,
\inn_{\prec}(g_e) \rangle = M = \langle \inn_{\prec}(f_0), 
\ldots, \inn_{\prec}(f_e)\rangle$ and 
$\{f_0, \ldots, f_e\}$ is a 
Gr{\"o}bner basis for $I_{\prec}(M)$  by Theorem~\ref{t:Evain}, 
it follows that $\{g_0, \ldots, g_e\}$
is a Gr{\"o}bner basis for $I_{\prec}(M)$.  The only monomials
occurring in $g_i$ are $m_i$ and standard monomials, so
$\{g_0,\ldots,g_e\}$ is a reduced Gr\"obner basis.
\end{proof}

\begin{lemma}
	\label{prop:reduction}
	Fix $m \in M$.  Let $\cT_m$ denote the set of monomials $s
        \not \in M$ of the same degree as $m$ with $s \prec m$.  Then
	\[ m \equiv \sum_{s\in \cT_m} \left( \sum _{\substack{ \cS 
	\text{ stroll from } \\ m \text{ to } s}}a_{\cS}\right) s 
	\mod I_{\prec}(M).\]
\end{lemma}

\begin{proof}
We proceed by induction on the number of monomials $m' \in
        M$ with $\deg(m')=\deg(m)$ and $m' \preceq m$.   The base case is when $m$ is the
        smallest monomial of its degree in $M$ with respect to $\prec$.  Note
        that in this case 
	all monomials occurring in $\frac{m}{m_{j(m)}}g_{j(m)}$
        other than $m$ are standard monomials, so we have
	\[m -  g_{j(m)}\frac{m}{m_{j(m)}} = -
	\sum_{s=mm'/m_{j(m)}, m'  \in \cT_{m_{j(m)}}} \left( \sum _{\substack{ \cW \text{ walk
              from } \\ m_{j(m)} \text{ to } m'}}a_{\cW}\right) s.\] Note that
        since $m$ is the smallest monomial of its degree in $M$, a
        stroll from $m$ to a standard monomial $s$ is the same as a
        walk, and the base case follows.  Now suppose that the claim
        is true whenever there are fewer smaller monomials in $M$ of
        the same degree.   Then 
	\begin{equation}\label{eq:1}
		m -  g_{j(m)}\frac{m}{m_{j(m)}} = 
- \sum_{m'=\frac{m}{m_{j(m)}}m'',  m'' \in \cT_{m_{j(m)}}}
	\left( \sum _{\substack{ \cW
	\text{ walk from } \\ m_{j(m)} \text{ to } m'
	\frac{m_{j(m)}}{m} }}a_{\cW}\right) m',\end{equation} 
	where 
	$m' \prec m$. If $m'$ is standard, then 
	a walk from $m$ to $m'$ is a stroll. If $m'
	\in M$, by induction we have
	\begin{equation}\label{eq:2}
		m' 
	\equiv \sum_{s\in \cT_{m'}} \left( \sum _{\substack{ \cS 
	\text{ stroll from } \\ m' \text{ to } s}}a_{\cS}\right) s 
	\mod I_{\prec}(M).\end{equation}
	Now a walk from $m_{j(m)}$ to $m'\frac{m_{j(m)}}{m}$ 
	occurring in \eqref{eq:1} combines with a stroll from 
	$m'$ to $s$
	occurring in \eqref{eq:2} to give a stroll from $m$ 
	to $s$.  As  every stroll occurs in this 
	way, the claim follows. 
\end{proof}

\begin{theorem}\label{thm:edgeequations}
	 Fix a $\mathbb Z^2/\mathbb Z \mathbf{c}$-grading and a term
         order $\prec$ on $\K[x,y]$.  Let $M , N \subset \K[x,y]$ be
         monomial ideals with Hilbert function $H$.  Suppose $M>N$
         with respect to the partial order of
         Definition~\ref{d:posetorder} induced by $\prec$.  
         For a minimal generator $n$  of $N$, and  $s$ a
         standard monomial for $M$ with $\deg(m) = \deg(s)$,
	 let
	\[F_{(n,s)} = \sum_{\substack{ \cH \text{ hike from }
	n \text{ to } s}} a_{\cH} 
	\in \K[T_+(M), T_-(N)].\]

	Then the ideal for the edge-scheme $E(M,N)$ parameterizing
        one-dimensional torus orbits containing $M$ and $N$ is given
        by
	\begin{equation*}
		I(M,N)    =
		\langle F_{(n,s)} \mid 
 n \text{ a
minimal generator 
of }N, s \not \in M, \deg(n)=\deg(s) 
	\rangle.
\end{equation*}
\end{theorem}

\begin{proof}
The theorem follows from applying the ideas of 
	\cite[Algorithm 5]{AltmannSturmfels} 
	 to the reduced Gr{\"o}bner basis $\{g_i\}$ for 
	 $I_{\prec}(M)$
	 of Lemma \ref{l:reducedbasis}  and using Lemma \ref{prop:reduction} to   
	 reduce the Gr{\"o}bner basis $\{\widetilde{f_i}\}$ 
	 for $I_{\prec^{\opp}}(N)$.

The edge scheme $E(M,N)$ is the scheme-theoretic intersection
$C_{\prec}(M) \cap C_{\prec^{\opp}}(N)$.  This equals the fiber product
  $C_{\prec}(M) \times_{\Hc(H)} C_{\prec^{\opp}}(N)$.  Thus to show
  that $E(M,N) = \Spec(\K[T_+(M), T_-(N)]/I(M,N))$, it suffices to show
  that the subscheme $C'$ of  \\ $\Spec(\K[T_+(M),T_-(N)]) = C_{\prec}(M)
  \times C_{\prec^{\opp}}(N)$ defined by $I(M,N)$ equals this fiber
  product.

To do this, we show that $C'$ satisfies the universal property of the
fiber product.  Indeed, let $i_M$ and $i_N$ be the inclusion morphisms
of $C_{\prec}(M)$ and $C_{\prec^{\opp}}(N)$ into $\Hc(H)$.  Suppose two
morphisms $\phi_M : \Spec(R) \rightarrow C_{\prec}(M)$ and $\phi_N :
\Spec(R) \rightarrow C_{\prec^{\opp}}(N)$ satisfy $i_M \circ \phi_M =
i_N \circ \phi_N$.  Let $\phi : \Spec(R) \rightarrow
\Spec(\K[T_+(M),T_-(N)])$ be the product $\phi_M \times \phi_N$.  We
will show that $\phi^*\left(F_{(n,s)}\right)=0$ for all pairs $(n,s)$, which
shows that $\phi$ factors through $C'$ with 
$\phi_M = p_1\circ \phi$ and $\phi_N =p_2 \circ \phi$, 
as required.  

To see this, let $I_M \subseteq R[x,y]$ be the ideal of the pull-back
of the universal family on $C_{\prec}(M)$ to $\Spec(R)$, and let $I_N
\subseteq R[x,y]$ be the ideal of the pull-back of the universal
family on $C_{\prec^{\opp}}(N)$ to $\Spec(R)$.  
The ideal $I_M$ is
generated by the polynomials $\phi_M^*(g_i)$ and 
the ideal $I_N$ is similarly 
generated by $\phi_N^*(f_i)$.
Our assumption that the induced maps to $\Hc(H)$ coincide imply that
$I_M=I_N$.

Fix a minimal generator $n=n_i$ of $N$.  Let $\cN_i= \{ n' \in
\Mon(\K[x,y]) \mid n'\prec^{\opp} n_i \text{ and } \deg(n') =
\deg(n_i)\}$.

Then for the generator $\tilde{f}_i$ of $I_{\prec^{\opp}}(N)$, we have 
\begin{IEEEeqnarray*}{rCl}
\widetilde{f_i} &=& 
n_i + \sum_{n\in \cN_i}\left( \sum_{\substack{\cP
	\text{ path from } \\n_i \text{ to }n\text{ for } \prec^{\opp}} 
	} a_{\cP}\right) n \\
&	=  &   
\sum _{\substack{s \text{ standard monomial} \\ \text{ 
for }M,   \deg(s) = \deg(n_i) }}
\left(\sum_{\substack{\cH \text{ hike 
from }\\ n_i  \text{ to } s}} a_{\cH}\right)s 
 + g,
\end{IEEEeqnarray*}
where $g\in I_{\prec}(M)$.  

By the definition of the multigraded Hilbert scheme,  the $R$-module
$(R[x,y]/I_M)_a$ is locally free of rank $H(a)$, where $a =
\deg(n_i)$.  Since $M \in \Hc(H)$, there are $H(a)$ standard monomials
of $M$ of degree $a$.  We claim that the standard monomials $s$ of $M$
of degree $a$ span $(R[x,y]/I_M)_a$.  If not, there is $\tilde{f}= \sum
c_m m \in R[x,y]_a$, where $c_m \in R$ and the $m$ are monomials that
are not in the span of the standard monomials modulo $I_m$.  We can choose
this polynomial to have $m'=\max_{\prec} \{ m : c_m \neq 0, m \in M
\}$ as small as possible.  Choose $i$ so that $m_i$ divides $m'$.
Then $\tilde{f} - c_{m'} m'/m_i \phi_M^*(g_i)$ is a polynomial with
smaller such maximum, so can be written in the desired form.  But then
the fact that $\phi_M^*(g_i) \in I_M$ gives a contradiction.  For a
prime $P$ of $R$ the $R_P$-module $(R_P[x,y]/R_PI_M)_a$ is free of
rank $H(a)$, so the spanning monomials must be a basis, and thus have
no relations between them.  
Since $\phi_N^*(\tilde{f}_i) \in I_N=I_M$
and $\phi_M^*(g) \in I_M$, we have $ \sum \phi^*(F_{(n_i,s)}) s \in
I_M$, where the sum is as above over the set of $s \not \in M$ with
$\deg(s) = \deg(n_i)$.

This means that the image of $\sum
\phi^*(F_{n_i,s,})s \in R_P[x,y]$ in each localization at  prime $P$
of $R$ must vanish, and so $\sum \phi^*(F_{n_i,s}) s =0$ in $R[x,y]$.  This means that
$\phi^*(F_{n_i,s})=0$ as required.  
\end{proof}

\begin{example}
Let $M= \langle x^8, x^5y, x^3y^3,y^4\rangle$ be the ideal from 
Example~\ref{ex:thebigexample} and let $N= \langle x^8, x^5y, x^2y^2,
y^6 \rangle$.  We order the generators for $N$ with respect to
$\prec^{\opp}$, and denote the significant arrows by by
$\widetilde{c}_i^{\ell}$.  We have $n_0 = y^6, n_1 = x^2y^2, n_2=x^5y,
n_3= x^8$, and the significant arrows are $\widetilde{c}_1^{1},
\widetilde{c}_1^2, \widetilde{c}_2^4$.  So the paths from $x^5y$ are $
\widetilde{c}_1^{1}, \widetilde{c}_1^2, \widetilde{c}_2^4$.  Then
		\[F_{(x^5y,x^6)} = c_1^1 +
		\widetilde{c_1^2}c_2^3 + \widetilde{c_1^2}c_2^1c_1^1c_1^1 + \widetilde{c_2^4} S_{xy^5,x^6}.\]

	Not all polynomials look this complicated. For example, 
	$F_{(x^2y^2,xy^3)} = \widetilde{c}_1^1+\widetilde{c}_1^2(
	c_3^1+c_2^1+c_1^1)$.
\end{example}
\begin{remark}
There can be significant cancellation in the equations $F_{(n,s)}$ for
$E(M,N)$ given in Theorem~\ref{thm:edgeequations}.  It would be
interesting to have a positive formula for these polynomials.  For
example, this would let us approach the question, asked in
\cite{AltmannSturmfels}, of whether the $T$-graph depends on the
characteristic of the field when $\K$ is algebraically closed.
\end{remark}

\section{Examples}

In this section we consider four different examples of $T$-graphs on
multigraded Hilbert schemes.

\subsection{The $T$-graph of $\Hilb^4(\mathbb A^2)$}.\label{s:n4}

The first nontrivial case of the $T$-graph of a multigraded Hilbert
scheme is the Hilbert scheme of four points in $\mathbb A^2$.  The
torus-fixed points in this case correspond to monomial ideals in
$\K[x,y]$ with four standard monomials, or equivalently to partitions
of four.  There are five such ideals/partitions:
\begin{enumerate}
\item $4$: $\langle x^4,y \rangle$,
\item $3+1$: $\langle x^3, xy, y^2 \rangle$,
\item $2+2$: $\langle x^2, y^2 \rangle$,
\item $2+1+1$: $\langle x^2,xy,y^3\rangle$,
\item $1+1+1+1$: $\langle x,y^4 \rangle$.
\end{enumerate}

\begin{figure}
\center{\resizebox{!}{5cm}{\includegraphics{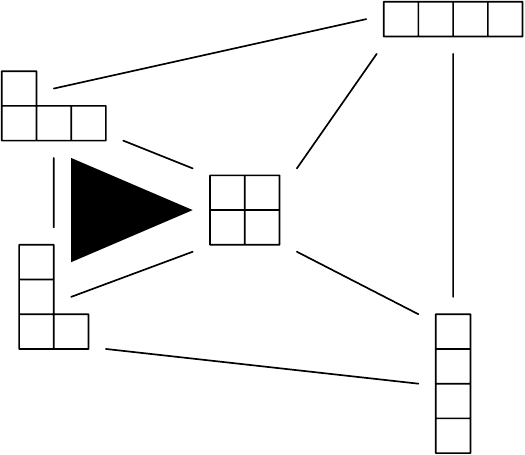}}}
\caption{\label{f:graphn4} The $T$-graph of $4$ points in $\mathbb
  A^2$.}
\end{figure}

Using this numbering, the edges of this graph are then $$\{ (1,2), (1,
3), (1, 5), (2, 3), (2, 4), (3, 4), (3, 5), (4, 5)\}.$$ Most of these
edges are one-dimensional.  We list them in the following table,
where the parameter $a$ can be any nonzero element of $\K$.
\begin{center}
\begin{tabular}{lll}
Edge & $\mathbf{c}$ & ideal \\ \hline
$(1,2)$ & $(1,-3)$ &  $\langle x^4, y-ax^3  \rangle$  \\
 $(1, 3)$& (1,-2)  & $\langle x^4,y-a x^2\rangle$  \\ 
$(1, 5)$&  (1,-1)  & $\langle x^4,y-ax   \rangle$ \\ 
$(2, 3)$&  (1,-1)  & $\langle  x^3,xy-ax^2,y^2  \rangle$ \\ 
$(3, 4)$&   (1,-1) & $\langle x^2,y^2-axy  \rangle$ \\ 
$(3, 5)$&   (2,-1) & $\langle x^2,y^2-ax   \rangle$ \\ 
$(4, 5)$&   (3,-1) & $\langle x^2,xy, y^3-ax  \rangle$ \\ 
\end{tabular}
\end{center}

The interesting edge is the one joining the second and fourth ideals
indicated by the black triangle.  It consists of the ideals
$$\langle x^3, xy-ax^2, y^2-(a^2+b)x^2 \rangle,$$ for any value of
$a,b \in \K^*$ with $a \neq b^2$.  Note that the edges $(2,3)$ and
$(3,4)$ live in the closure of this edge (by letting $b=-a^2$ or $a
\rightarrow \infty$ respectively).  This is shown in Figure~\ref{f:graphn4}.

\subsection{The $T$-graph of two  points in $\mathbb P^2$}

We now consider the case of the Hilbert scheme of two points in
$\mathbb P^2$.  As a multigraded Hilbert scheme this corresponds to
requiring the Hilbert function to be $h(0)=1$, $h(1)=3$, and $h(i)=2$ for $i
\geq 2$ for ideals in $\K[x_0,x_1,x_2]$.  The $T$-graph has nine
vertices, which we label by the saturations of the corresponding
monomial ideals, as these have fewer generators.  Explicitly, for an
ideal $M \in \Hilb^h_S$, we label the vertex by $(M: \langle
x_0,x_1,x_2 \rangle^{\infty})$.  We can recover $M$ from its
saturation by taking the ideal generated by the degree-two part of the
saturation.  The nine saturated ideals are:
\begin{center}
\begin{tabular}{llllllll}
1 & $\langle x_0,x_1x_2 \rangle $  & \hspace{1cm} & 2 & $\langle x_1,x_2x_3 \rangle $ & \hspace{1cm}  & 3 & $\langle x_2,x_0x_1 \rangle $ \\
4 & $\langle x_0,x_2^2 \rangle $ & \hspace{1cm}  & 5 & $\langle x_0,x_1^2 \rangle $ & \hspace{1cm}  & 6 & $\langle x_1,x_0^2 \rangle $ \\
7 & $\langle x_1,x_2^2 \rangle $ & \hspace{1cm}  & 8 & $\langle x_2,x_1^2 \rangle $ & \hspace{1cm}  &9 & $\langle x_2,x_0^2 \rangle $ \\
\end{tabular}
\end{center}
The $T$-graph then has $18$ edges.  Up to the $S_3$-symmetry, these
are the pairs: $\{ (\langle x_0,x_1x_2 \rangle, \langle x_1,x_0x_2
\rangle), (\langle x_0,x_1x_2 \rangle , \langle x_0,x_1^2 \rangle),
(\langle x_0,x_1^2 \rangle, \langle x_1,x_0^2 \rangle), (\langle
x_0,x_1^2 \rangle, \langle x_0,x_2^2 \rangle ),$ \\ $ (\langle x_0,x_1^2
\rangle, \langle x_2, x_1^2 \rangle) \}$.  This is shown in
Figure~\ref{f:2ptsP2}.  The shaded triangles indicate that the edges
joining ideals $4$ and $5$, joining $6$ and $7$, and joining $8$ and $9$ are
two dimensional, and have the third vertex of the respective triangles
in their closure.  For example, ideals in the edge joining vertices $4$ and $5$
have the form $\langle x_0^2,x_0x_1,x_0x_2,x_1^2+ax_1x_2+bx_2^2
\rangle$ for $a,b \in K$ with $b \neq 0$.

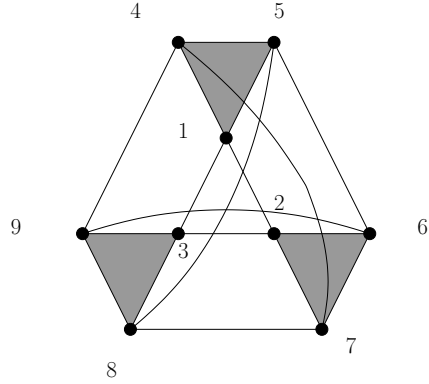
\begin{figure}
\center{\resizebox{!}{5cm}{\input{2ptsinP2.pdftex_t}}}
\caption{\label{f:2ptsP2} The $T$-graph of $2$ points in $\mathbb
  P^2$.}
\end{figure}

\subsection{The $T$-graph of $\Hilb^8(\mathbb A^4)$}.  

The Hilbert scheme of $8$ point in $\mathbb A^4$ has two irreducible
components, of dimensions 32 and 25, which intersect in a scheme 
of dimension 24.  See \cite{CEVV} for more details.  The $T$-graph of
this Hilbert scheme has 684  vertices, and 9278 edges.  All vertices
lie on the component of dimension 32.  This data can be found in the
package {\tt TEdges}~\cite{TEdges}.

\subsection{Small Hilbert schemes of points in the plane}

We list here the data for the Hilbert schemes $\Hilb^d(\mathbb A^2)$
for small values of $d$.  This illustrates the use of our necessary
condition in this range.

\begin{center}
\begin{tabular}{lllllll}\label{table}
$d$ & \# ideals   & \# pairs & \# pairs & \# pairs $(M,N)$  & \# pairs $(M,N)$    & \# edges \\ 
& $M$  & $(M,N)$ & $M<N$ &  with an arrow  & with an arrow map & \\ 
&  & & & map  & on the duals& \\ \hline
4 & 5 &10 &8 &8 & 8 &8 \\
5 & 7& 21 &15 &15 &15 & 15\\
6 & 11&55 & 37&37 &37 &37 \\
7 & 15 & 105 &55 &52 &52 &52 \\
8 &22 & 231& 100& 99 &99 &99 \\
9 &30 & 435 &170 &166 &166 &166 \\
10 & 42& 861 &291 & 280 &280 &280 \\
11 & 56& 1540 &411 &401 & 401 &401 \\
12 & 77& 2926& 688 &663 &663 & 663\\
13 & 101& 5050 &957 & 918& 918 & 918 \\
14 & 135& 9045 &1524  & 1446 & 1446& 1446\\
15 & 176 & 15400 & 2203 & 2076 &2076 & 2076\\
16 & 231 & 26565&  3218& 3033 & 3031 & 3031 \\
\end{tabular}
\end{center}


\begin{remark}
This table was created with the Macaulay 2 package
{\tt{TEdges}}~\cite{TEdges}.  For $d\leq 15$ the edge code was run
independently from the partial order and arrow-map code, so the
containment of the set of edges in the set for which there exist
arrow-maps gives a check on the code.  This was not possible for
$d>15$ for memory usage reasons.  While this table shows that the
necessary conditions of Theorem~\ref{thm:main} are sufficient for
small $d$ in $\Hilb^d(\mathbb A^2)$, we caution that 
$16$ points is
still comparatively small for this problem, so do not regard this as
strong evidence of the condition being sufficient, particularly in
light of Example~\ref{ex:almostall}.
\end{remark}

\bibliographystyle{alpha}
\bibliography{HilbertSchemePoints}
\end{document}

%% file: 2ptsinP2.pdftex_t
\begin{picture}(0,0)%
\includegraphics{2ptsinP2.pdftex}%
\end{picture}%
\setlength{\unitlength}{3947sp}%
\begingroup\makeatletter\ifx\SetFigFont\undefined%
\gdef\SetFigFont#1#2#3#4#5{%
  \reset@font\fontsize{#1}{#2pt}%
  \fontfamily{#3}\fontseries{#4}\fontshape{#5}%
  \selectfont}%
\fi\endgroup%
\begin{picture}(5100,4714)(2101,-5819)
\put(4201,-2761){\makebox(0,0)[lb]{\smash{{\SetFigFont{12}{14.4}{\rmdefault}{\mddefault}{\updefault}{\color[rgb]{0,0,0}\Large{$1$}}%
}}}}
\put(5401,-3661){\makebox(0,0)[lb]{\smash{{\SetFigFont{12}{14.4}{\rmdefault}{\mddefault}{\updefault}{\color[rgb]{0,0,0}\Large{$2$}}%
}}}}
\put(4201,-4261){\makebox(0,0)[lb]{\smash{{\SetFigFont{12}{14.4}{\rmdefault}{\mddefault}{\updefault}{\color[rgb]{0,0,0}\Large{$3$}}%
}}}}
\put(3601,-1261){\makebox(0,0)[lb]{\smash{{\SetFigFont{12}{14.4}{\rmdefault}{\mddefault}{\updefault}{\color[rgb]{0,0,0}\Large{$4$}}%
}}}}
\put(5401,-1261){\makebox(0,0)[lb]{\smash{{\SetFigFont{12}{14.4}{\rmdefault}{\mddefault}{\updefault}{\color[rgb]{0,0,0}\Large{$5$}}%
}}}}
\put(7201,-3961){\makebox(0,0)[lb]{\smash{{\SetFigFont{12}{14.4}{\rmdefault}{\mddefault}{\updefault}{\color[rgb]{0,0,0}\Large{$6$}}%
}}}}
\put(6301,-5461){\makebox(0,0)[lb]{\smash{{\SetFigFont{12}{14.4}{\rmdefault}{\mddefault}{\updefault}{\color[rgb]{0,0,0}\Large{$7$}}%
}}}}
\put(3301,-5761){\makebox(0,0)[lb]{\smash{{\SetFigFont{12}{14.4}{\rmdefault}{\mddefault}{\updefault}{\color[rgb]{0,0,0}\Large{$8$}}%
}}}}
\put(2101,-3961){\makebox(0,0)[lb]{\smash{{\SetFigFont{12}{14.4}{\rmdefault}{\mddefault}{\updefault}{\color[rgb]{0,0,0}\Large{$9$}}%
}}}}
\end{picture}%